\pdfoutput=1
\documentclass[a4paper, 12pt]{amsart}
\usepackage{setspace}
\usepackage[inner=1in,outer=1in,tmargin=1in,bmargin=1in]{geometry}
\usepackage{mathtools}
\usepackage{amssymb}
\usepackage[utf8]{inputenc}
\usepackage[T1]{fontenc}
\usepackage[english]{babel}
\usepackage{parskip}
\usepackage{tensor}
\usepackage{hyperref}
\usepackage{amsmath}
\usepackage{afterpage}
\usepackage{tikz}
\usetikzlibrary{decorations.markings}
\usepackage{ae,aecompl}
\usepackage{graphicx}
\usepackage[capposition=bottom]{floatrow}
\usepackage{empheq}
\usepackage{amsthm}
\usepackage{lipsum} 
\usepackage{cite}
\usepackage{cases}
\usepackage{hyperref}
\usepackage{pgfplots}
\pgfplotsset{compat=1.18}
\usepackage{enumitem}
\usepackage{pdfpages}
\usepackage{slashed}
\usepackage[colorinlistoftodos]{todonotes}
\usepackage{comment}

\allowdisplaybreaks

\hypersetup{
	colorlinks,
	linkcolor={red!80!black},
	citecolor={blue!80!black},
	urlcolor={blue!80!black},
	linktoc=page
}

\usetikzlibrary{decorations.pathmorphing,calc}
\makeatletter

\usepackage{mathtools}
\mathtoolsset{showonlyrefs}

\makeatother
\DeclareMathOperator{\Tr}{Tr}

 \providecommand{\norm}[1]{\lVert#1\rVert}

\newcommand{\rup}{^\rho}

\newcommand{\mud}{_{\mu}}

\newcommand{\nud}{_{\nu}}

\newcommand{\munud}{_{\mu\nu}}

\newcommand{\nn}{\nonumber}

\newcommand{\pt}{\partial}

\newcommand{\be}{\begin{equation}}
	\newcommand{\ee}{\end{equation}}
\newcommand{\ba}{\begin{eqnarray}}
	\newcommand{\ea}{\end{eqnarray}}
\newcommand{\baa}{\begin{array}}
	\newcommand{\eaa}{\end{array}}

\newcommand{\na}{\nabla}

\newcommand{\M}{\mathcal{M}}
\newcommand{\g}{\overline{g}}
\newcommand{\bg}{\Box_{\g, A,Q}}
\newcommand{\T}{\mathcal{T}}
\renewcommand{\H}{\mathcal{H}}
\renewcommand{\Im}{\textnormal{Im}}
\newcommand{\ct}{\mathfrak{c}}
\newcommand{\tchi}{\Tilde{\chi}}

\newcommand{\bd}{b_\eta}
\newcommand{\rr}{\mathcal{R}}

\newcommand{\pminus}[1]{^{-#1}}

\newcommand{\taux}{{\kappa(x,v)}}

\makeatletter
\newcommand\incircbin
{%
  \mathpalette\@incircbin
}
\newcommand\@incircbin[2]
{%
  \mathbin%
  {%
    \ooalign{\hidewidth$#1#2$\hidewidth\crcr$#1\bigcirc$}%
  }%
}

\makeatother

\theoremstyle{definition}

\theoremstyle{plain}
\newtheorem{theorem}{Theorem}[section]
\newtheorem{lemma}[theorem]{Lemma}
\newtheorem{proposition}[theorem]{Proposition}

\newtheorem{corollary}[theorem]{Corollary}

\theoremstyle{definition}
\newtheorem{definition}{Definition}[section]

\theoremstyle{definition}
\newtheorem{question}{Question}[section]
\newtheorem{remark}{Remark}[section]

\newcommand{\Norm}[1]{\Big\rVert #1 \Big\rVert}

\newcommand{\innerproduct}[2]{\langle #1, #2 \rangle}

\newcommand{\LR}{\mathcal{L}}

\newcommand{\R}{\mathbb{R}}

\def\C{\mathbb C}
\newcommand{\N}{\mathbb{N}}
\newcommand{\Cinf}{C^\infty}
\newcommand{\Cinfo}{C^\infty_0}

\newcommand{\B}{\mathcal{B}}

\newcommand{\A}{\mathcal{A}}

\newcommand{\ola}{\overline{A}}
\newcommand{\olm}{\overline{M}}

\newcommand{\ww}{\Box_{A,Q}}

\newcommand{\intr}[1]{\operatorname{Int} #1}

\def\p{\partial}
\DeclareMathOperator{\supp}{supp}

\begin{document}

\begin{abstract}
We study the problem of recovering a time dependent matrix valued potential on a globally hyperbolic manifold from the knowledge of the source to solution map of a wave equation including a connection 1-form term. We exhibit sufficient conditions for solving this inverse problem under the assumption that the the manifold is stationary and that the connection term is time independent. The proof is based on two ingredients. The first is reduction of the problem to the study of a non-Abelian light ray transform and holds assuming global hyperbolicity only. The second is the study of this transform and establishing a link with a Riemannian analogue. 
\end{abstract}

\title[Recovery of a matrix valued potential]{Recovery of a matrix valued potential for the wave equation on stationary spacetimes}

\author[S.FILIPPAS, L. OKSANEN, M. SARKKINEN]{Spyridon Filippas, Lauri Oksanen, Miika Sarkkinen}

\address{Department of Mathematics and Statistics, University of Helsinki, PO BOX 68, 00014, HELSINKI, Finland}
\email {spyridon.filippas@helsinki.fi}

\address{Department of Mathematics and Statistics, University of Helsinki, PO BOX 68, 00014, HELSINKI, Finland}
\email {lauri.oksanen@helsinki.fi}

\address{Department of Mathematics and Statistics, University of Helsinki, PO BOX 68, 00014, HELSINKI, Finland}
\email {miika.@sarkkinen.fi}

\maketitle
\tableofcontents

\section{Introduction}
In this article we deal with the problem of recovering the~\textit{time dependent, matrix valued} potential of a wave equation from the knowledge of the source to solution map. Similar questions have been recently studied in~\cite{mishra2021determining, kumar2024stabledeterminationtimedependentmatrix} where the authors prove uniqueness and stability results respectively for the recovery of a matrix valued potential in a Euclidean setting from partial boundary observations. We are interested in treating a similar problem in the general setting of globally hyperbolic geometries. The notion of global hyperbolicity gives the right class of Lorentzian manifolds guaranteeing existence and uniqueness of global solutions to the wave equation. We start by recalling some important definitions and facts that will be needed to state and prove our main results.

\subsection{Facts from Lorentzian geometry and setting of the problem} 
For an introduction to Lorentzian geometry we refer to~\cite{ONeill, ringstrom2009}. We consider $(\M, \g)$ a connected Lorentzian manifold of dimension $1+n$ with signature $(-,+,...,+)$. A vector $v \in T_p \M \backslash \{0\}, \: p \in \M$, is timelike, spacelike or null if
\begin{equation}
    g(v,v)<0,\quad g(v,v)>0, \quad g(v,v)=0,
\end{equation}
respectively. We say that $v \in T_p \M \backslash \{0\} $ is a causal vector if it is timelike or null, that is $g(v,v) \leq 0$. The manifold $(\M, \g)$ will be assumed to be time oriented. That means that there exists a vector field $T$ such that $T(p)$ is timelike for all $p \in \M$. This implies that for a causal vector $v$ one has $g(v,T)<0$ or $g(v,T)>0$. Then $v$ is said to be future-directed if $g(v,T)<0$ and past-directed if the opposite inequality holds. 

 A submanifold $S$ of $\M$ is called spacelike if all the non-zero tangent vectors of $S$ are spacelike. A smooth curve $\gamma$ is timelike if its tangent vector $\dot \gamma(s)$ is timelike for all $s \in [a,b]$ with analogous definitions for spacelike, null, causal and future/past directed. 

Given $p,q \in \M$ we write $p \leq q$ if $p=q$ or there is a future-directed causal curve from $p$ to $q$. This allows to define the~\textit{causal future} of a subset $A \subset \M$ as the set
\begin{equation}
    J^+(A)=\{q \in \M \textnormal{: there is $p$} \in S \textnormal{ such that } p\leq q \},
\end{equation}
with an analogous definition for the~\textit{causal past} $J^-(A)$ of $A$, for which the inequality above is reversed. We write $J^{\pm}(p)= J^{\pm}(\{p\})$ for points $p \in \M$. We can now define global hyperbolicity.
\begin{definition}
    The manifold $(\M, \g)$ is~\textit{globally hyperbolic} if
    \begin{enumerate}
        \item \label{point_one} there is no closed causal curve,

        \item \label{point_two} the causal diamonds $J^+(p) \cap J^-(q)$ are compact sets for all $p, q \in \M$. 
    \end{enumerate}
\end{definition}

We refer to~\cite{hounnonkpe2019} for this relatively recent definition. If $n \geq 2$ and $\M$ is non-compact then~\eqref{point_one} can be omitted from the definition above~\cite{hounnonkpe2019}. On the other hand, globally hyperbolic manifolds are non-compact. Global hyperbolicity has far reaching consequences, which are summarized in the next theorem. A hypersurface $\Sigma \subset \M$ is called a~\textit{Cauchy surface} if every inextendible timelike curve intersects $\Sigma$ exactly once.

\begin{theorem}
\label{prop_of_globhyper}
    The following are equivalent for $(\M, \g)$.

\begin{enumerate}
    \item $(\M, \g)$ is globally hyperbolic.

    \item There exists a smooth spacelike Cauchy surface in $\M$.

    \item There exists a temporal function, that is a smooth function $\tau : \M \to \R$ that satisfies $g(d \tau,d \tau)<0$, $d \tau $ is future-directed, and all level sets $\Sigma_t=\tau^{-1}(t)$ are smooth spacelike Cauchy surfaces. The function $\tau$ can be taken surjective and such that its restriction to any inextendible future-directed causal curve is strictly increasing and surjective onto $\R$. Moreover, $\M$ is diffeomorphic to $\R \times \Sigma$ via  $F :\R \times \Sigma_0 \to \M$ with
    \begin{equation}
        (F^*g)(t,x)=c(t,x)(-dt^2+h_t(x)),
    \end{equation}
 where $c>0$ is smooth, $h_t(x)$ is a Riemannian metric on $\Sigma_0$ smoothly depending on $t$ and $t=\tau$.    
\end{enumerate}
\end{theorem}
We note that the diffeomorphism $F$ is obtained by flowing $\Sigma_0$ along $-\nabla^{\g} \tau$, where $\nabla^{\g}$ denotes the gradient operator on $(\M, \g).$

This equivalence has been established in~\cite{bernal2003, bernal2005}. See also the references in~\cite[Proposition 2.3]{oksanen2024rigiditylorentziancalderonproblem} from which comes the formulation in Theorem~\ref{prop_of_globhyper}.

\medskip

 We write $\Box_{\g}$ for the Laplace-Beltrami operator on $\M$ defined by $\Box_{\g}= - \textnormal{div}_{\g} \nabla^{\g}u$. For $N \in \N^*=\{1,2...\}$ we consider $A \in C^\infty (\M; \C^{N \times N} \otimes T^*\M)$ a matrix valued one-form and $Q \in C^\infty_0 (\M; \C^{N \times N})$ a matrix valued potential. We then define the operator $\bg$ acting on vectors $u(t,x)=(u_1(t,x),u_2(t,x)...u_N(t,x))^T$, $u_j \in C^\infty (\M), N \in \N^*$ by
\begin{equation}
    \bg u =\Box_{\g} u +2 A \cdot \nabla^{\g } u + Q u=\begin{pmatrix}
\bg u_1   \\
\bg u_2 \\
. \\
. \\
\bg u_N
\end{pmatrix}+  2\sum_{i,j=0}^n \g^{ij} A_i\begin{pmatrix}
\p_j u_1  \\
\p_j u_2 \\
. \\
. \\
\p_j u_N
\end{pmatrix}+Q \begin{pmatrix}
u_1  \\
u_2 \\
. \\
. \\
u_N
\end{pmatrix}.
\end{equation}
As already mentioned, globally hyperbolic manifolds provide with a general geometric setting in which the Cauchy problem for the wave equation is well posed. This is the content of the following proposition.
\begin{proposition}
\label{prop_wave_lorentz}
  Let $(\M, \g)$ be globally hyperbolic. Then for any $f \in C^\infty_0(\M)$ there exists a unique $u \in C^\infty (\M)$ such that
 \begin{equation}
 \label{eq_wave_lorentz}
\begin{cases}
    \bg u=f  \quad \textnormal{in } \M,\\
    \supp (u) \subset J^+(\supp (f)).
    \end{cases}
\end{equation}
\end{proposition}

For the existence part of the above proposition we refer to~\cite[Theorem 3.3.1]{bar2007} and the uniqueness follows from~\cite[Theorem 3.1.1]{bar2007}

Let $K \subset \M$ denote a fixed compact set. With Proposition~\ref{prop_wave_lorentz} at our disposition we can now define in an intrinsic way the source to solution map associated to the system~\eqref{eq_wave_lorentz}. Given $f \in C^\infty_0 (\M \backslash K)$ we write $u^f$ for the unique solution of~\eqref{eq_wave_lorentz} with source $f$. This allows to define a map
\begin{align}
    \rr: C^\infty_0 (\M \backslash K) &\to C^\infty (\M \backslash K), \\
    f &\to u^f_{ | \M \backslash K },
\end{align}
where $_{| \M \backslash K}$ denotes the restriction to $\M \backslash K$. Physically speaking, we suppose that we have access to the zone $\M \backslash K$ where we can use sources $f$ and measure the resulting waves. We want to know whether these measurements can determine a matrix valued potential $Q$ supported in $K$. In other words, we ask the following.

\begin{question}
\label{question_base}
 Consider $Q_j \in C^\infty_0 (\M; \C^{N \times N})$ with $\supp(Q_j) \subset K$, $j \in \{1,2\}$. Denote by $\rr_j$ the source to solution maps associated to $Q_j$, $j \in \{1,2\}$. Does $\rr_1=\rr_2$ imply that $Q_1=Q_2$?
\end{question}

\subsection{Main results}
Our first result establishes a link between the recovery of the potential $Q$ and a non-Abelian light ray transform and it holds without any additional assumption on the globally hyperbolic manifold $\M$. We write $P_{A, \gamma}$ for the parallel transport along $\gamma$ for the connection $-A$ (see Section~\ref{sec_reduction_glob_hyper} for a precise definition).
\begin{theorem}
    \label{thm_reduction_lorentz}
Let $\gamma : [a,b] \to \M$ be a null geodesic on $\M$. Suppose that the endpoints $\gamma(a), \gamma(b) \notin K$ lie outside of the support of $Q$. Then the source to solution map $\rr$ uniquely determines
 \begin{equation}
  \label{eq_light_ray_thm}
    \int_{\gamma(a)}^{\gamma(b)} P_{A,\gamma}(s)^{-1} Q(\gamma(s))P_{A,\gamma}(s)ds.
\end{equation}
\end{theorem}

\medskip

With the help of Theorem~\ref{thm_reduction_lorentz} our main problem reduces to that of studying a non-Abelian light ray transform. As it turns out, we are not able to answer to Question~\ref{question_base} in this level of generality. We are however able to exhibit a sufficient condition allowing to answer it in the affirmative assuming that $\M$ is in addition~\textit{stationary} and $A$ is~\textit{time-independent}. A spacetime is called stationary provided that there exists a complete timelike Killing vector field. Recall that a vector field is Killing if its flow preserves the metric. The flow of the Killing vector field gives rise to a global time coordinate. Since $\mathcal{M}$ is assumed to be globally hyperbolic, one can choose a Cauchy slice $\Sigma$ provided by Theorem~\ref{prop_of_globhyper} and obtain a time coordinate using the Killing vector flow. For the precise definitions, see Section \ref{sec-stat}. Time dependence or independence of objects on the Lorentzian manifold $\M$ is to be understood with respect to the time coordinate associated with the Killing field. For a time-independent connection in a stationary spacetime, the parallel transport $P_{A,\gamma}$ reduces to parallel transport along the projection of a null geodesic onto $\Sigma$.

We assume as well that $K$, the support of $Q$, is such that $K \subset \R \times \textnormal{Int}(M) \subset \M $ with $M$ an $n-$ dimensional Riemannian manifold with smooth boundary. The metric on $M$ is naturally related to the Lorentzian metric $\overline{g}$. Here the splitting $\R \times \textnormal{Int}(M)$ is such that the metric on $\mathcal M$ reads 
\begin{equation}
    \overline{g} = -(dt + \omega)^2 + g,
\end{equation}
where the 1-form $\omega$ and the Riemannian metric $g$ are time independent. A null geodesic of the above metric can be naturally described in terms of its projection on $M$, which is a magnetic geodesic $x(s)$ that satisfies the Lorentz force equation
\begin{equation}\label{eq_lorentz_force}
    \na_{\dot x}\dot x = F(\dot x).
\end{equation}
Here the bundle map $F: TM \to TM$ is given by the magnetic field, $F^i_j = -(d\omega)^i_j$. The stationary Lorentzian metric $\overline g$ thus gives rise to a magnetic system $(g,\omega)$ on $M$. We write $\varphi_s(x,v)$ for the flow of the unit speed magnetic geodesic starting at $x \in M$ in the direction $v$ having $[0,\taux]$ as maximal interval of definition, and we assume that $\varphi$ is nontrapping, that is, the exit time $\taux$ is a finite real number. 

Notice that if $A$ is time independent, it naturally defines a connection on the Riemannian manifold $M$. We denote by $SM$ the unit sphere bundle of $M$ and define a non-Abelian magnetic X-ray transform of a matrix valued function $w \in C^\infty(SM, \C^{N \times N})$ with $\supp w \subset \textnormal{Int } (SM )$ by
\begin{equation}
   I_\A w(x,v) = \int_0^\taux P_{\A}(s,x,v)\pminus{1}w(\varphi_s(x,v)) P_{\A}(s,x,v) ds, \quad \forall (x,v) \in \pt_+ SM.
\end{equation}
We write $P_{\A}(s,x,v)$ for the smooth map given by the parallel transport along $\varphi$ for the connection $-A$. We refer to Section~\ref{sec-stat} where all the above mentioned objects are defined and constructed in a detailed manner.

\medskip

Injectivity of $I_\A$ refers to the property
\begin{equation}
\label{property_injecivity_xray}
I_\A w=0 \textnormal{ for all }(x,v) \in \pt_+ SM \Longrightarrow w \equiv 0.
\end{equation}
In Theorem~\ref{thm_sufficient_condition_for_light_ray} we show that the light ray transform~\eqref{eq_light_ray_thm} is injective if $I_\A$ is injective. This gives the following.
\begin{theorem}
\label{thm_main}
    Suppose that the globally hyperbolic manifold $\M$ is stationary and that the connection form $A$ is time independent. Suppose that $\supp Q \subset \R \times \textnormal{Int}( M) $ with $M$ an $n-$dimensional Riemannian manifold with strictly magnetic convex boundary. Then injectivity of $I_\A$ implies that $\rr$ determines $Q$.
\end{theorem}
Above strictly magnetic convex boundary is defined with respect to the magnetic system $(g,\omega)$. For a precise definition, we refer to \eqref{eq_magnetic_convex}.

\begin{remark}
 The assumption on the support of $Q$ is not particularly restrictive and essentially asks that the projection of $K$ on the Cauchy slice $\Sigma$ is strictly contained in $\Sigma$ (see Figure~\ref{fig_support}). This is for instance automatically satisfied if $\Sigma$ is non-compact. 
\end{remark}

The proof of Theorem~\ref{thm_reduction_lorentz} is based on a Gaussian beam construction. Gaussian beams are approximate solutions of the wave equation concentrating near a null geodesic. With Theorem~\ref{thm_reduction_lorentz} at our disposal Question~\ref{question_base} reduces to the study of a non-Abelian light ray transform. We show that the light ray transform can be analyzed in terms of a transport problem on the tangent bundle of a stationary Lorentzian manifold. Using time independency we apply Fourier slicing techniques to reduce this into a transport problem on the unit sphere bundle of a constant time slice. We show that the new transport problem is related to the non-Abelian magnetic X-ray transform. The sufficient condition for the inversion of the light ray transform is obtained in Theorem~\ref{thm_sufficient_condition_for_light_ray}, which combined with Theorem~\ref{thm_reduction_lorentz} proves Theorem~\ref{thm_main}.

\medskip

To make the abstract assumption in our main theorem more tangible, it is useful to have some concrete examples where the assumption is guaranteed to hold. We formulate these examples into the following corollaries.

\begin{corollary}\label{cor_A_nonzero}
    Suppose that $A$ is a time-independent unitary connection on a trivial bundle. Then $\mathcal{R}$ determines $Q$ if either of the following holds.
    \begin{enumerate}
        \item $M$ is a surface and the magnetic system $(g,\omega)$ is simple.
        \item The Lorentzian manifold $\overline M$ is static, $\dim (M) \geq 3$, $\pt M$ is strictly convex, and $M$ admits a strictly convex foliation.
    \end{enumerate}
\end{corollary}
We prove this in Section \ref{sec_non_abelian_light}.
\begin{corollary}\label{cor_A_zero}
    Suppose that $A = 0$ and $\dim (M) \geq 2$. Then $\mathcal{R}$ determines $Q$ provided that $M$ is simply connected, has a strictly magnetic convex boundary, i.e. a boundary satisfying \eqref{eq_magnetic_convex}, and $\varphi$ is nontrapping and free of conjugate points.
\end{corollary}
\begin{proof}
    With $A=0$, the parallel transport map $P_\A$ becomes trivial, and the transform $I_\A w$ reduces to the magnetic X-ray transform of each matrix entry of $w$. The claim then simply follows from \cite[Thm. 5.3]{dairbekov2007boundary} and \cite[Thm. 1.2]{Ains_13} combined with our Theorem~\ref{thm_main}.
\end{proof}

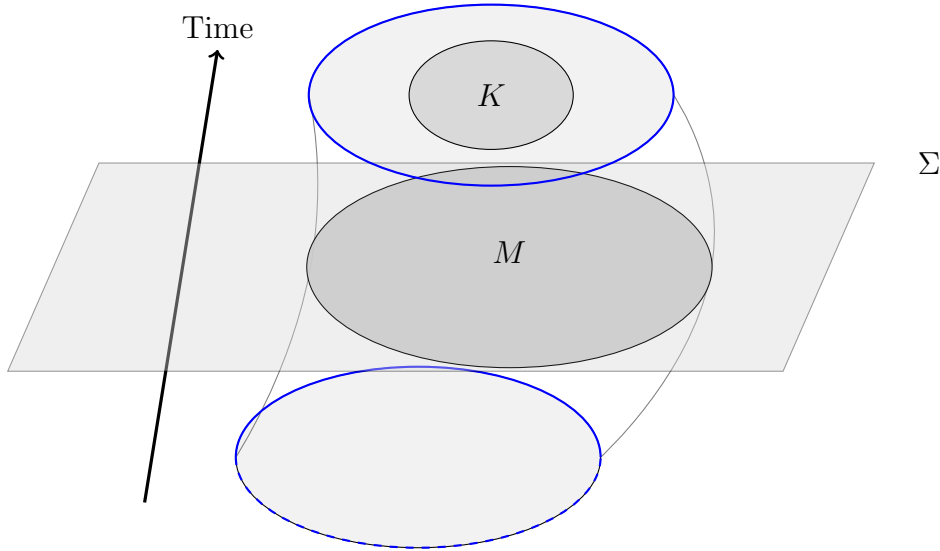
\begin{figure}
\centering
\begin{tikzpicture}[scale=1.2]

\def\rx{2}   
\def\ry{1}   
\def\h{4}    
\def\tilt{0.8} 

\coordinate (B) at (0,0);
\coordinate (T) at (\tilt,\h);

\draw[fill=gray!10] (B) ellipse ({\rx} and {\ry});
\draw[fill=gray!10] (T) ellipse ({\rx} and {\ry});

\draw[fill=gray!30] (T) ellipse ({0.9} and {0.6});
\node[black!50!black] at (T) {$K$};

\draw[gray] plot [smooth, tension=1] coordinates {($(B)+(-\rx,0)$) (-1.2,2) ($(T)+(-\rx,0)$)};
\draw[gray] plot [smooth, tension=1] coordinates {($(B)+(\rx,0)$) (3.2,2) ($(T)+(\rx,0)$)};

\draw[thick,blue] ($(B)+(\rx,0)$) arc (0:180:{\rx} and {\ry});
\draw[thick,blue,dashed] ($(B)+(-\rx,0)$) arc (180:360:{\rx} and {\ry});

\draw[->,very thick] ($(B)+(-3,-0.5)$) -- ($(T)+(-3,0.5)$) node[above] {Time};

\def\sliceY{2.1}     
\def\sliceWidth{8}   
\def\sliceHeight{0.3} 

\coordinate (S1) at ($(-\sliceWidth/2+0.5,\sliceY+\sliceHeight/2+1)$);
\coordinate (S2) at ($( \sliceWidth/2+1,\sliceY+\sliceHeight/2+1)$);
\coordinate (S3) at ($( \sliceWidth/2-0.,\sliceY-\sliceHeight/2-1)$);
\coordinate (S4) at ($(-\sliceWidth/2-0.5,\sliceY-\sliceHeight/2-1)$);

\draw[fill=gray!30,opacity=0.4] (S1) -- (S2) -- (S3) -- (S4) -- cycle;

\node[black!50!black] at ($(S2)+(0.6,0)$) {$\Sigma$};

\coordinate (MC) at (1,\sliceY); 
\draw[fill=gray!40,opacity=0.9] (MC) ellipse ({\rx/0.9} and {\ry/0.9});
\node at ($(MC)+(0,-\ry/1.2 + 1)$) {$M$};

\draw[thick,blue] ($(T)+(\rx,0)$) arc (0:180:{\rx} and {\ry});
\draw[thick,blue] ($(T)+(-\rx,0)$) arc (180:360:{\rx} and {\ry});

\end{tikzpicture}
\caption{We assume that the support of $Q$ is contained in some spacetime cylinder with the time coordinate given by the flow of the Killing vector field.}
\label{fig_support}
\end{figure}

\subsection{Previous literature} One of the main approaches for attacking hyperbolic inverse problems is the Boundary Control method, originating in the seminal work of Belishev~\cite{belishev1987approach} in the context of the isotropic wave equation. It allows to recover the geometry or potential from the knowledge of the Dirichlet to Neumann map or variants under very general assumptions on the Riemannian factor of the hyperbolic equation. The main ideas of Boundary Control method are used in~\cite{kurylev2018a} in which the authors recover a matrix valued potential (and some other quantities) from a partial Dirichlet to Neumann map associated to the wave equation of the connection Laplacian. This result holds under general geometric assumptions, however the potential is assumed to be \textit{time independent}. This is due to the fact that Boundary Control method uses in a crucial way the unique continuation result of~\cite{tataru1995} for which time independency or at least analyticity with respect to time for all of the coefficients of the operator is needed. Recently, a variation of the Boundary Control method was used in~\cite{alexakis2022} allowing to attack the Lorentzian Calderón problem for time dependent potentials that are only \textit{smooth} in time. However, in those references, other assumptions are imposed on the Lorentzian manifold, the most important of which is some curvature bounds which would not hold for instance in the presence of positive curvature. Concerning time dependent coefficient problems let us as well mention~\cite{doi:10.1137/21M1400596} in which the stability of a formally determined hyperbolic inverse problem was obtained based on a modification of the Bukhgeim-Klibanov method~\cite{bukhgeim1981global}.

The approach we follow in the paper is based on the construction of geometric optics solutions. Those allow to make a link between time dependent coefficient determination problems for the wave equation and the light ray transform on Lorentzian manifolds. In the setting of Minkowski spacetime, invertibility of the light ray transform was first proved for scalar functions in~\cite{stefanov1989uniqueness}. Since then, geometric optics solutions (or some generalizations like Gaussian beams) have been used in various geometric contexts, reducing an inverse problem to the study of a light ray transform. For a (non exhaustive) list of works on this approach see for instance~\cite{ANIKONOVROMANOV+1997+487+490, aicha2015stability, kian2019a, feizmohammadi2021b, stefanov2018} as well as~\cite{feizmohammadi2021light, oksanen2025interplay, lassas2020} for works mainly focused on the invertibility of the light ray transform. Let us point out that the strength of our result with respect to the previous literature relies on the fact that the potential is~\textit{time dependent and matrix valued}, and that the wave equation includes a connection term on a globally hyperbolic setting.

\subsection{Outline of the paper}
The proof of Theorem~\ref{thm_reduction_lorentz} is the content of Section~\ref{sec_reduction}. We start by giving a proof in the Minkowski case, relying on geometric optics solutions. This allows to follow the main ideas of the proof in a simpler geometric context. We then give the general proof. In Section~\ref{sec-stat} we give some useful facts about stationary manifolds. Next, we study the non-Abelian light ray transform in stationary spacetimes in Section~\ref{sec_non_abelian_light}. We conclude with two appendices containing some useful geometric facts.

\section{Reduction to a non-Abelian light ray transform}
\label{sec_reduction}

\subsection{The Minkowski case}

As a warm-up, we give in this section a proof of an analogue of Theorem~\ref{thm_reduction_lorentz} in the simpler setting of Minkowski geometry. That is, the globally hyperbolic manifold is $\R^{1+n}$ equipped with the Minkowski metric $g_{\textnormal{Min}}=-dt^2+dx^2$. This allows to understand the idea of the proof of Theorem~\ref{thm_reduction_lorentz} with less technicalities. We will assume here that the potential is supported in some cylinder in space-time. More precisely, on $\R^{1+n}$ we consider the wave operator 
\begin{equation}
    \ww=\Box + 2 A \cdot \nabla+Q,
\end{equation}
where we write $\Box= \p^2_t- \Delta$ for the standard wave operator and $A$ is a time dependent matrix valued 1-form 
\begin{equation}
A=A_0(t,x) dt+ \sum_{i=1}^n A_i(t,x)dx^i, \quad A_j \in \C^{N \times N}, N \in \N^*.
\end{equation}
The matrix valued potential $Q=Q(t,x) \in \C^{N \times N}$ is time dependent and satisfies $\supp (Q) \subset \R \times \Omega$ with $\Omega \subset \R^n$ bounded. Given a displacement vector $$u(t,x)=(u_1(t,x),u_2(t,x),...,u_N(t,x))^T,$$ we consider the operator $\ww$ acting as
\begin{align}
\ww u &= \Box u +2  \sum_{j=0}^n A_j \p_j u+Q u  =  \begin{pmatrix}
\Box u_1   \\
\Box u_2 \\
. \\
. \\
\Box u_N
\end{pmatrix}+  2\sum_{j=0}^n A_j\begin{pmatrix}
\p_j u_1  \\
\p_j u_2 \\
. \\
. \\
\p_j u_N
\end{pmatrix}+Q \begin{pmatrix}
u_1  \\
u_2 \\
. \\
. \\
u_N
\end{pmatrix},
\end{align}
with $\p_0= \p_t$ and $\p_j=\p_{x_j}, $ $j \in \{1,2...,n\}$. For a vector valued source $f \in C^\infty_0((0,T) \times (\R^n \backslash \Omega))$ we consider the solution of the system
\begin{equation}
\begin{cases}
    \ww u=f & \textnormal{in} \: (0,T) \times \R^n,\\
    u_{|t=0}= \p_t u_{|t=0}=0  & \textnormal{in} \: \R^n.
    \end{cases}
\end{equation}
This allows to define the source to solution map
\begin{equation}
    f \mapsto Rf:= u_{|x \notin \Omega}.
\end{equation}
We define $P_{A, \gamma}(s)$ as the normalized fundamental matrix solution to the system of ODEs
\begin{equation}
            \p_s u - (A \cdot  \dot \gamma )u=0.
\end{equation}
That is $P_{A, \gamma}(s) \in \C^{N \times N }$ satisfies
\begin{equation}
\label{def_transport_minkowski}
     \p_s P_{A, \gamma}(s)- (A \cdot  \dot \gamma )P_{A, \gamma}(s)=0, \quad P_{A, \gamma}(0)=\textnormal{Id}.
\end{equation}

Recall that a geodesic $\gamma$ is null if it has light-like tangent vector $\dot \gamma$. In this section we prove the following.

\begin{proposition}
\label{prop_minkow_reduction}
Let $\gamma : [0, \ell] \to (0,T) \times \R^n$ be a null geodesic on $\R^{1+n}$ for the Minkowski metric. Assume that the endpoints $\gamma(0), \gamma (\ell ) \notin \Omega$ lie outside of $\Omega$. Then the source to solution map $R$ uniquely determines 
 \begin{equation}
    \int_0^\ell P_{A,\gamma}(s)^{-1} Q(\gamma(s))P_{A,\gamma}(s)ds.
\end{equation}
\end{proposition}
In the statement above we suppose that the time component of the geodesic $\gamma$ lies on the interval $(0,T)$. This assumption can be always satisfied by choosing $T>0$ large enough.

\medskip

\textit{Proof of Proposition~\ref{prop_minkow_reduction}.} Remark that null geodesics in the Minkowski geometry are precisely straight lines of the form 
$$\gamma(s)=(s+t_0,s v+x_0), \quad s \in \R,$$ 
with $(t_0,x_0) \in \R^{1+n}$ and $v \in \R^n$ such that $|v|=1$. Up to a rotation and translation we may assume without loss of generality that $\gamma(s)=(s,s,0,...,0)$. We now construct some geometric optics solutions. That is solutions of the form $e^{i\sigma \phi}a_{\sigma}$ with 
\begin{equation}
\label{def_of_a}
a_\sigma=a_0+\sigma^{-1}a_1+\sigma^{-2}a_2,
\end{equation}
where $\sigma>1$ is a large parameter, $\phi$ a real valued phase function and $a_j=a_j(t,x)=(a^1_j(t,x),a^2_j(t,x),...,a^N_j(t,x))^T \in \C^{N}$. We will chose $a_j$ so that 
\begin{equation}
    \ww e^{i \sigma \phi} a_\sigma =O(\sigma^{-2}),
\end{equation}
or equivalently
\begin{equation}
 e^{-i \sigma \phi}   \ww e^{i \sigma \phi} a_\sigma =O(\sigma^{-2}).
\end{equation}
We start by calculating
\begin{align}
  e^{-i \sigma \phi}   \ww e^{i \sigma \phi} a_0
  &=(\Box+Q)a_0+2i\sigma\left( \partial_t \phi \partial_t -\nabla_x \phi \cdot \nabla_x +\frac{1}{2} \Box \phi \right)a_0-\sigma^2 \left(|\partial_t \phi|^2-  |\nabla_x \phi|^2\right)a_0 \\
  &\hspace{4mm} +2 i \sigma (A \cdot \nabla \phi) a_0+2 A \cdot \nabla a_0,
\end{align}
where the $j$-th component of the vector $\nabla_x \phi \cdot \nabla a_0 \in \C^N$ is given by $\nabla_x \phi \cdot \nabla_x a_0^j$. 

This forces in particular the following eikonal equation for the phase $\phi$ 
\begin{equation}
   |\partial_t \phi|^2- |\nabla_x \phi|^2=0, 
\end{equation}
and we can take $\phi=t-x^1$. With this choice for $\phi$ and looking at the terms multiplied by $\sigma$ above we obtain for $a_0$
\begin{equation}
    (\partial_t+\partial_{x^1}+A \cdot \nabla \phi)a_0=0,
\end{equation}
which after the change of variables
\begin{equation}
    s=\frac{t+x^1}{2}, \quad r=\frac{t-x^1}{2},
\end{equation}
becomes
\begin{equation}
\label{eq_for_transport}
            \p_sa_0- (A \cdot  \dot \gamma )a_0=0.
\end{equation}
Note that in $s,r$ coordinates we have $\gamma(s)=(s,0)$. Here, $\cdot$ denotes the pairing between vectors and covectors. We fix $x_0 \in \C^N\backslash\{0\}$ and consider a cut-off $\chi =\chi(y)\in C^\infty_0(\Sigma)$, $\chi(0)=1$, with $\Sigma$ the hyperplane normal to $\gamma$ through the origin. We then chose
\begin{equation}
    a_0(s)=P_{A, \gamma}(s) \chi(y) x_0,
\end{equation}
where $P_{A, \gamma}$ is defined in~\eqref{def_transport_minkowski}. In particular, we have $a_0(0)=x_0 \chi(y)$.
With this choice for $a_0$ we expand the terms in $e^{-i \sigma \phi}   \ww e^{i \sigma \phi} a_\sigma$ to find
\begin{align}
    e^{-i \sigma \phi}   \ww e^{i \sigma \phi} a_\sigma &=\Box_Q a_0+2 A \cdot \nabla a_0+\sigma^{-1}\Box_Qa_1+2i (\p_t+\p_{x_1})a_1+2i (A\cdot \nabla \phi) a_1\\
    &\hspace{4mm}+2\sigma^{-1} A \cdot \nabla a_1 + \sigma^{-2}\Box_Qa_2+2i \sigma^{-1} (\p_t+\p_ {x_1})a_2+2i \sigma^{-1}(A\cdot \nabla \phi) a_2\\
    &\hspace{4mm}+2\sigma^{-2} A \cdot \nabla a_2,
\end{align}
with $\Box_Q=\Box+Q$. The above quantity is in $O(\sigma^{-2})$ whenever $a_1$ and $a_2$ satisfy the following equations
 \begin{align}
     \p_s a_j- (A \cdot \dot\gamma )a_j=\frac{i }{2}\Box_Q a_{j-1}+i A \cdot \nabla a_{j-1}, \quad j \in \{1,2\}.
 \end{align}
 We can moreover impose the initial condition
 \begin{equation}
     a_j(0)=0, \quad j \in \{1,2\}.
 \end{equation}
Let us now consider the solution $b_1$ of the ODE
\begin{equation}
    \p_s b_1- (A \cdot \dot\gamma )b_1=\frac{i }{2}\Box a_{0}+i A \cdot \nabla a_{0},\quad b_1(0)=0.
\end{equation}
We then define 
\begin{equation}
\label{def_of_c1}
c_1=\frac{2}{i}(a_1-b_1),
\end{equation}
and observe that $c_1$ satisfies the equation
\begin{equation}
\label{eq_for_c1}
    \p_s c_1- (A \cdot \dot \gamma)c_1=Qa_0.
\end{equation}
Since $a_1(0)=0$ we also have $c_1(0)=0$. We now remark that the solution of~\eqref{eq_for_c1} satisfying $c_1(0)=0$ can be written as
\begin{equation}
\label{expression_for_c1}
    c_1(s)=P_{A,\gamma}(s)\int_0^s P_{A,\gamma}(s^\prime)^{-1} Q(\gamma(s^\prime))P_{A,\gamma}(s^\prime)ds^\prime a_0(0).
\end{equation}
Indeed, the derivative of the right hand side of~\eqref{expression_for_c1} with respect to $s$ is given by
\begin{align}
 &\left(\frac{d}{ds} P_{A,\gamma}(s)\right)\int_0^s P_{A,\gamma}(s^\prime)^{-1} Q(\gamma(s^\prime))P_{A,\gamma}(s^\prime)ds^\prime a_0(0) + P_{A,\gamma}(s)  P_{A,\gamma}(s)^{-1} Q(\gamma(s))P_{A,\gamma}(s) a_0(0)\\
 &=(A \cdot \dot \gamma)P_{A,\gamma}(s)\int_0^s P_{A,\gamma}(s^\prime)^{-1} Q(\gamma(s^\prime))P_{A,\gamma}(s^\prime)ds^\prime a_0(0) +Q(\gamma(s))P_{A,\gamma}(s) a_0(0)\\
 &=(A \cdot \dot \gamma)c_1+Q(\gamma(s))a_0,
\end{align}
thanks to the definitions of $P_{A,\gamma}$ and $a_0$.

The last step for the proof of Proposition~\ref{prop_minkow_reduction} is to show that the knowledge of the source to solution map $R$ allows to recover $c_1(\ell)$. If this is the case we also recover $P_{A,\gamma}(\ell)^{-1}c_1(\ell)$ and hence
\begin{equation}
    \int_0^\ell P_{A,\gamma}(s^\prime)^{-1} Q(\gamma(s^\prime))P_{A,\gamma}(s^\prime)ds^\prime x_0 \chi(y).
\end{equation}

We observe that $b_1$ is independent of $Q$. Consequently, recalling the definition of $c_1$ in~\eqref{def_of_c1} it suffices to prove that $R$ determines $a_1(\ell)$. This is the following lemma.

\begin{lemma}
\label{lem_minkow_Rdetermines_sol}
Let $\gamma : [0, \ell] \to (0,T) \times \R^n$ be a null geodesic and assume that $\gamma(0), \gamma (\ell ) \notin \Omega$. Then, the source to solution map $R$ uniquely determines $a_1(\ell)$.
\end{lemma}

\begin{proof}
We consider a cut-off $\tchi= \tchi(t) \in C^\infty(\R )$ such that $\tchi=1$ for $t<0$ and $\tchi=0$ for $t>\eta$ with $\eta>0$ small. We define then $r_\sigma$ as the solution to the following wave equation
\begin{equation}
\begin{cases}
    \ww r_\sigma =-(1-\tchi) \ww e^{i \sigma \phi} a_\sigma & \textnormal{in} \: (0,T) \times \R^n,\\
    r_{\sigma|t=0}= \p_t u_{\sigma|t=0}=0  & \textnormal{in} \: \R^n.
    \end{cases}
\end{equation}
We define $w=(1-\tchi) e^{i \sigma \phi}a_\sigma+ r_\sigma$ with $a_\sigma$ as in~\eqref{def_of_a} with the choices of $a_j$ described above. Then $w$ satisfies  
\begin{equation}
\begin{cases}
    \ww w =[\tchi, \ww] e^{i \sigma \phi} a_\sigma & \textnormal{in} \: (0,T) \times  \R^n,\\
     w_{|t=0}= \p_t w_{|t=0}=0  & \textnormal{in} \: \R^n.
    \end{cases}
\end{equation}
The important thing about the term $[\tchi, \ww]$ is that it is a differential operator with coefficients supported on $\supp (\tchi ^\prime)$, that is on the strip $0<t< \eta$. The construction of $a_j, j \in \{0,1,2\}$ implies that they are supported close to $\gamma$, hence $a_\sigma$ is supported close to $\gamma$ as well. As a consequence, up to choosing the supports of $\tchi^\prime$ and $\chi$ small enough we have that $[\tchi, \ww] e^{i \sigma \phi} a_\sigma$ is supported close to $\gamma (0)$. The assumption $\gamma(0) \notin \Omega$ combined with the fact that $Q$ is supported in $\Omega$ gives then $[\tchi, \ww]a_\sigma= [\tchi, \Box_{A,0}]a_\sigma$ and we see that the term $[\tchi, \ww] e^{i \sigma \phi} a_\sigma$ is a smooth function supported away from $\Omega$ and independent of $Q$. As a consequence the source to solution map $R$ determines $w_{| x \notin \Omega}$, and therefore $R$ uniquely determines
\begin{align}
    \sigma e^{- i \sigma \phi } (1- \tchi)e^{i \sigma \phi}(a_0+\sigma^{-1}a_1+\sigma^{-2}a_2)+ \sigma e^{- i\sigma \phi } r_\sigma,
\end{align}
for $x \notin \Omega$. Since $\ww e^{i \sigma \phi}a_\sigma=O(\sigma^{-2})$ energy estimates for the wave equation imply that $r_\sigma(t)= O(\sigma^{-2})$ in $H^1$. Recalling that $a_0$ is independent of $Q$ we get that for any $\psi \in C^\infty_0 ( (0,T) \times (\R^n \backslash \Omega))$ the operator $R$ determines
\begin{equation}
    \left((1- \tchi)(a_1+ \sigma^{-1}a_2)+\sigma e^{- i \sigma \phi}r_\sigma , \psi \right)_{L^2((0,T)   \times (\R^n \backslash \Omega))} \to \left((1- \tchi)a_1, \psi \right)_{L^2((0,T)   \times (\R^n \backslash \Omega))},
\end{equation}
as $\sigma \to \infty$. Assuming that $\gamma (\ell) \notin \Omega$ this proves that $R$ determines $a_1 (\ell)$ as stated.
\end{proof}

Varying the choices of $x_0$ along a basis of $\C^N$ completes the proof of Proposition~\ref{prop_minkow_reduction}.

\subsection{The case of globally hyperbolic geometries}
\label{sec_reduction_glob_hyper}
We now turn our attention to the case of globally hyperbolic Lorentzian geometries. Notice that by Property (3) of Theorem~\ref{prop_of_globhyper} we have that $(\M ,\g)$ is isometric to $(\R \times \Sigma, \g)$ where
\begin{equation}
    \g=c(t,x)(-dt^2+h_t(x)), \quad
    t \in \R,  \: x \in \Sigma=\tau^{-1}(0),
\end{equation}
and $\tau$ is the temporal function of Theorem~\ref{prop_of_globhyper}. Here $c$ is a smooth positive function and $h_t$ is a family of Riemannian metrics on $\Sigma$, smoothly depending on $t$.

Similarly to the Minkowski case, we want to construct an approximate solution to the equation $\bg u =0$ which concentrates close to a null geodesic $\gamma$. The main difficulty in this general geometric setting is that geometric optics solutions, that is approximate solutions of the form $e^{i \sigma \phi }a$ with \textit{real valued} phase $\phi$, do not always exist. However, if one allows the phase to be \textit{complex valued} one can get an adequate replacement of geometric optics solutions: the Gaussian beams for the wave equation. Their construction is classical and was introduced in~\cite{ralston1982gaussian, babich1984complex}. In the context of inverse problems they appear first time in~\cite{belishev1992boundary}.

In~\cite{feizmohammadi2022}, they are constructed for a scalar wave equation in globally hyperbolic Lorentzian manifolds. Our exposition follows that of~\cite[Section 4]{feizmohammadi2022}. We start by recalling the basic properties of the~\textit{Fermi coordinates}, near a null geodesic $\gamma$. For a proof of the following lemma we refer to~\cite[Lemma 1]{feizmohammadi2022}. 
\begin{lemma}[Fermi coordinates]
\label{lem_fermi}
Let $\delta >0$, $a<b$ and $\gamma : [a-\delta, b+ \delta] \to \M$ be a null geodesic on $\M$. There is a coordinate neighborhood $(U, \Phi)$ of $\gamma([a,b])$, with coordinates denoted by $(z^0=s,z^1,...,z^n)$, such that:
\begin{itemize}
    \item $\Phi(U)=(a-\delta^\prime,b+\delta^\prime)\times B(0, \delta^\prime)$, where $B(0, \delta^\prime)$ denotes a ball in $\R^n$ with a small radius $\delta^\prime$.

    \item  $\Phi(\gamma(s))=(s,0..,0)$.
\end{itemize}
Moreover, the metric tensor $\g$ satisfies in this coordinate system
\begin{equation}
\label{metric_in_fermi}
    \g_{|\gamma}=2 ds \otimes dz^1+\sum_{j=2}^n dz^j \otimes dz^j,
\end{equation}
and ${\p_i {\g_{jk}}}_{|\gamma}=0$ for $i,j,k=0,...,n$ where $|_{\gamma}$ denotes the restriction on $\gamma$.    
\end{lemma}

In the sequel we shall use the notation $(z^1,z^2,...,z^n)=z^\prime$, for the variables that are transversal to the curve $\gamma$ in the Fermi coordinates.
\bigskip

We now fix a null geodesic $\gamma : [a-\delta, b+ \delta] \to \overline{M}$ and work in the coordinate system described in Lemma~\ref{lem_fermi}. As before we will consider an ansatz of the form $e^{i \sigma \phi} a_\sigma (s, z^\prime)$ but this time $\phi: \M \to \C$ is allowed to take complex values. The phase $\phi$ and the amplitude $a$ will be constructed for $s \in I=[a- \delta^\prime, b+ \delta^\prime]$ and $|z^\prime| < \delta^\prime$ with $\delta^\prime$ small.

We have for the conjugated operator
\begin{equation}
    e^{-i \sigma \phi} \bg e^{i \sigma \phi}a=\sigma^2 (\H \phi) a-i \sigma \T a+ \bg a.
\end{equation}
The operators $\H : C^\infty(\M) \to C^\infty(\M)$ and $\T \in C^\infty(\M; \C^N)\to C^\infty(\M; \C^N) $ are given by 
\begin{equation}
    \H \phi= \langle d \phi, d \phi \rangle_{\g}, \quad \T a =2 \langle d \phi , d a\rangle_{\g}-(\Box_{\g} \phi)a-2(A \cdot \nabla^{\g } \phi )a,
\end{equation}
where the $j$-th component of the vector $2 \langle d \phi , d a\rangle_{\g}-(\Box_{\g} \phi)a \in \C^N$ is given by 
\begin{equation}
    2 \langle d \phi , d a^j\rangle_{\g}-(\Box_{\g} \phi)a^j, \quad a=(a^1,...,a^N), \:  j \in \{1,2...,N\}.
\end{equation}
Similarly to the Minkowski case, we want to impose the eikonal equation $\H \phi =0 $ for the phase $\phi$. The important difference is that in this general geometric context this equation does not necessarily have a solution in a neighborhood of the geodesic (in particular in the presence of conjugate points for $\M$, see for instance the discussion in~\cite[Section 3.5]{feizmohammadi2021b} ). To remedy this problem, we will solve the equations only on $\gamma$ and use a Taylor approximation for the transversal directions. More precisely, we shall seek for solutions of the form
\begin{equation}
\label{def_of_phi_a}
  \phi= \sum_{k=0}^m \phi_k (s, z^\prime), \quad a_\sigma= \chi\left(\frac{|z^\prime|}{\delta^\prime}\right)\sum_{k=0}^m \sigma^{-k}a_k(s,z^\prime),  
\end{equation}
with
\begin{itemize}
    \item $\phi_k \in C^\infty(\M, \C)$, $k \in \{0,...,m\}$ complex-valued homogeneous polynomials of degree $k$ with respect to the variables $z^\prime=(z^1,...,z^n)$.

    \item $a_k= \sum_{l=0}^m a_{k,l} $ with $a_{k,l} \in C^\infty(\M;\C^N)$ such that for $j \in \{1,...,N\}$ and \break $k,l \in \{0,...,m\}$ the component $a^j_{k,l}$ is a complex-valued homogeneous polynomial of degree $l$ with respect to the variables $z^\prime$.
 \end{itemize}
The cut-off $\chi \in C^\infty_0(\R)$ satisfies $\chi(0)=1$ and localizes close to $\gamma$ for $\delta^\prime$ small. We impose the following equation for $\phi$ on $\gamma$
\begin{equation}
\label{eq_for_phi}
    \frac{\p^\alpha}{\p{z^\prime}^\alpha}(\H \phi)(s,0)=0, \quad \forall s \in I,
\end{equation}
for all multi-indices $\alpha$ with $|\alpha| \leq m$. Note that we write $(s,0)$ for the point $(s,z^\prime)$ with $z^\prime=0$. For the main amplitude $a_0$ we require that
\begin{equation}
\label{eq_for_a0}
     \frac{\p^\alpha}{\p{z^\prime}^\alpha}(\T a_0)(s,0)=0, \quad \forall s \in I,
\end{equation}
and for the next amplitudes $a_k$ with $k \in \{1,...,m\}$
\begin{equation}
\label{eq_for_ak}
     \frac{\p^\alpha}{\p{z^\prime}^\alpha}(-i \T a_k + \bg a_{k-1})(s,0)=0, \quad \forall s \in I,
\end{equation}
for all multi-indices $\alpha$ with $|\alpha| \leq m$. We then have the following definition.
\begin{definition}
    An approximate Gaussian beam of order $m$ along $\gamma$ is a function $e^{i \sigma \phi} a_\sigma \in C^\infty(\M;\C^N)$ with $\phi$ and $a_\sigma$ as in~\eqref{def_of_phi_a} satisfying additionally
    \begin{enumerate}
        \item Equations~\eqref{eq_for_phi}, \eqref{eq_for_a0} and \eqref{eq_for_ak} hold.

        \item $\Im(\phi)_{|\gamma}=0$, that is the imaginary part of $\phi$ vanishes along $\gamma$.

       \item
       \label{cond_phi}
       $\Im (\phi)(z) \geq C |z^\prime|^2$ for all points $(s, z^\prime) \in \M$ with $s \in I$ and $|z^\prime| < \delta^\prime$.
    \end{enumerate}
\end{definition}
The key point is that condition~\eqref{cond_phi} allows to get exponential decay and by a Taylor expansion one can prove (see~\cite[Lemma 2]{feizmohammadi2022} for a proof) that $e^{i \sigma \phi} a_\sigma$ is indeed an approximate solution.
\begin{lemma}
\label{lem_gauss_beam_decay}
Suppose that $e^{i \sigma \phi} a_\sigma$ is Gaussian beam of order $m$ along the geodesic $\gamma$. Consider $\tau_0, \tau_1 \in \R$ such that $\gamma(a), \gamma(b) \notin [\tau_0,\tau_1] \times \Sigma$. Then for all $\sigma >0$ one has
\begin{equation}
    \Norm{\left(\bg (e^{i \sigma \phi} a_\sigma)\right)^j}_{H^k((\tau_0,\tau_1) \times \Sigma)} \lesssim \sigma^{-K}, \quad j \in \{1,...,N\}.
\end{equation}
with $K=\frac{m+1}{2}+\frac{n}{4}-k-2$. Here $u^j$ denotes $j-th$ component of the vector $u \in C^\infty(\M;\C^N)$. 
\end{lemma}

\bigskip

Equation~\eqref{eq_for_phi} for the phase $\phi$ is exactly the same as in the scalar case and does not involve the connection term $A$. Consequently, we can use the same construction as in~\cite[Section 4.2.1]{feizmohammadi2022}. In particular, the important terms $\phi_0$ and $\phi_1$ will be chosen as
\begin{equation}
    \label{phi0_and_phi1}
    \phi_0=0, \quad \phi_1=z_1.
\end{equation}
It turns out that the terms $\p^\alpha_{z^\prime} \phi_2$ for $\alpha=2$ satisfy a (non linear) Riccati equation which admits a unique solution. Subsequent terms $\phi_j, j \geq 3$ satisfy systems of linear ODE's with coefficients depending on $\phi_l, \: l\leq j-1$. 

With $\phi$ already constructed we now turn our attention to the amplitude function $a_\sigma$. For $|\alpha| =0$, equation~\eqref{eq_for_a0} gives on $\gamma$
\begin{equation}
\label{eq_ongamma1}
2 \sum_{k,l=0}^n g^{kl}\p_{z^k} \phi \p_{z^l} a_0-(\Box_{\g} \phi)a_0-2 (A \cdot \nabla^{\g} \phi) a_0=0, \quad \forall s \in I.
\end{equation}
Here we used the fact that since $\nabla^{\g} \phi= \sum_{k,l=0}^n g^{kl} \p_{z^k} \phi \frac{\p}{\p z^l}$, thanks to~\eqref{metric_in_fermi} we have on $\gamma$ that 
\begin{equation}
  \nabla^{\g} \phi=\p_s \phi \frac{\p}{\p z_1}+ \p_{z^1} \phi \frac{\p}{\p s}+\sum_{k=2}^n  \p_{z^k} \phi \frac{\p}{\p z^k}.
\end{equation}
We have that $\p_s \phi_0 =0$ and $\p_s \phi_k$ for $k \geq 1$ vanish too on $\gamma=\{z^\prime=0\}$ since they are homogeneous polynomials of degree $k$ with respect to $z^\prime$. Similarly $\p_{z^k} \phi$ vanish on $\gamma$ for $k \geq 2$ and we finally get that $ \nabla^{\g} \phi= \frac{\p}{ \p s}$ on $\gamma$. Similarly, we find that $2 \sum_{k,l=0}^n g^{kl}\p_{z^k} \phi \p_{z^l} a_0=2 \frac{d}{ds} a_0$ and we see that~\eqref{eq_ongamma1} reduces to
\begin{equation}
\label{eq_ongamma2}
  \frac{d}{ds}  a_{0,0}-\frac{(\Box_{\g} \phi)}{2}a_{0,0}- (A \cdot \dot \gamma) a_{0,0}=0, \quad \forall s \in I.
\end{equation}
Here we also used the fact that since all the components of $a_{0,k}$ are homogeneous polynomials of order $k$ in $z^\prime$ all terms involving $a_{0,k}$ for $k \geq 1$ vanish on~\eqref{eq_ongamma1} in the equation for $a_0$ for $\alpha=0$. We shall denote in the sequel 
\begin{equation}
    \ct = -\frac{(\Box_{\g} \phi)}{2}.
\end{equation}
We fix then $x_0 \in \C^N \backslash \{0\}$ and define $a_{0,0}$ as the unique solution to the system of linear ODE's 
\begin{equation}
\label{eq_for_a0_lorentz}
  \frac{d}{ds}  u+ \ct u - (A \cdot \dot \gamma) u=0, \quad u(\gamma(a))=\chi\left(\frac{|z^\prime|}{\delta^\prime}\right)x_0,
\end{equation}
We can then construct the terms $a_{0,k}$ with $k \geq 1$ inductively by solving systems of linear ODE's with coefficients depending on $a_{0,j}$ with $0 \leq j \leq k-1$. To construct $a_{0,1}$ we consider equation~\eqref{eq_for_a0} for $|\alpha|=1$. This gives the equation
\begin{equation}
\label{eq_for_v01}
    \frac{\p^{\alpha}}{\p z^\prime} \left( 2 \sum_{k,l=0}^n g^{kl}\p_{z^k} \phi \p_{z^l} a_0-(\Box_{\g} \phi)a_0-2 (A \cdot \nabla^{\g} \phi) a_0\right)=0,
\end{equation}
on $\gamma$.
Remark that if $|\beta|=0$ or $|\beta|\geq 2$ then $\p^\beta_{z^\prime} a_{0,1}=0$ on $\gamma$ since all the components of $a_{0,1}$ are homogeneous polynomials of order $1$ with respect to $z^\prime$. As a consequence, the terms involving $a_{0,1}$ in the equation above are those of the form  $\p^{\beta}_{z^\prime} a_0$ for $|\beta|=1$ and the terms of the form $\p^\alpha_{z^\prime}\p_s a_0$ with $|\alpha|=1$. Using these observations along with~\eqref{metric_in_fermi}, we see that~\eqref{eq_for_v01} can be written as
\begin{equation}
    \frac{d}{ds} (\p^\alpha_{z^\prime} a_{0})+\sum_{|\beta|=1}\mathcal{A}_{\alpha \beta}\left(\p^\beta_{z^\prime} a_{0} \right)=\mathcal{B}_{\alpha}, \quad \forall s \in I,
\end{equation}
on $\gamma$. The crucial point is that the coefficients $\mathcal{A}_{\alpha \beta}=\mathcal{A}_{\alpha \beta}(s) \in C^\infty( I ;\C)$ and  $\mathcal{B}_\alpha=\mathcal{B}_\alpha(s) \in C^\infty (I; \C^N)$ only depend on $g, \phi, A, Q$ and $a_{0,0}$ which is already constructed. Let us denote by $\mathcal{N}_k$ the cardinality of the set $\{\alpha \textnormal{ such that } |\alpha|=k\}$. By taking all multi-indices $\alpha$ with $|\alpha|=1$ we obtain $N \cdot \mathcal{N}_{1}$  linear ODE's for $\p^\alpha_{z^\prime} a_{0,0}$ with known coefficients. This system of linear ODE's can be uniquely solved by imposing additionally the initial condition $\p^\alpha_{z^\prime} a_{0,0}(\gamma(a))=0$. To construct the subsequent terms $a_{0,k}$ one proceeds in a similar fashion, obtaining a system of linear ODE's for $\p^\alpha_{z^\prime} a_{0,1}$, $|\alpha|=k$, with $N \cdot \mathcal{N}_k$ unknowns and coefficients depending on $g, \phi, A, Q$ and $a_{0,j}$, $j \in \{0,...,k-1\}$.

This concludes the construction of the term $a_0=\sum_{j=1}^m a_{0,j}$. The procedure for constructing the terms $a_{j}$ with $j \geq 1$ follows in a similar fashion by solving equation~\eqref{eq_for_ak}. We will however need more information on the term $a_{1,0}$. We look at~\eqref{eq_for_ak} for $k=1$ and $\alpha = 0$. This gives on $\gamma$ 
\begin{equation}
\label{eq_for_a1}
    \frac{d}{ds} a_{1,0}+\ct a_{1,0}-(A \cdot \dot \gamma)a_{1,0}=-\frac{i}{2}\bg a_{0}.
\end{equation}
We construct then $a_{1,0}$ as the solution to the system of linear ODE's~\eqref{eq_for_a1} satisfying $a_{1,0}(0)=0$. Terms $a_{1,j}, ...a_{k,j}$ are constructed iteratively analogously to the terms $a_{0,j}$. 

We are now ready to prove the analogue of Proposition~\ref{prop_minkow_reduction} in the globally hyperbolic case. First we use the fact the source to solution map $\rr$ uniquely determines $a_{1,0}$ in the points where the geodesic $\gamma$ is not in $K$.

\begin{lemma}
\label{lem_lorent_Rdetermines_sol}
Assume that $\gamma(a), \gamma (b ) \notin K$. Then, the source to solution map $\rr$ uniquely determines $a_{1,0}(\gamma(b))$.
\end{lemma}
\begin{proof}
    The proof is similar to the one Lemma~\ref{lem_minkow_Rdetermines_sol}. Here we need to work in the coordinates given by Property (3) of Theorem~\ref{prop_of_globhyper}. In these coordinates $\M$ becomes $\R \times \Sigma$ with $\Sigma=\tau^{-1}(0)$ and $\tau$ the temporal function of Theorem~\ref{prop_of_globhyper}. 

    We remark that since $\gamma(a), \gamma (b ) \notin K$ one can choose $\eta>0$ small enough so that $\gamma(a\pm\eta), \gamma(b \pm \eta) \notin K$. With this choice of $\eta$ we pick $\tau_0, \tau_1 \in \R$ satisfying (see Figure~\ref{fig:gamma-curve})
    \begin{equation}
        \tau (\gamma(a))< \tau_0< \tau (\gamma(a+\eta)), \quad \tau(\gamma(b-\eta)) < \tau_1 < \tau(\gamma(b)).
    \end{equation}
We now proceed as in Lemma~\ref{lem_minkow_Rdetermines_sol}. We consider a cut-off $\tchi= \tchi(t) \in C^\infty(\R )$ such that $\tchi=1$ for $t<\tau_0$ and $\tchi=0$ for $t>\eta^\prime$ with $\eta^\prime>0$ small. We define then $r_\sigma$ as the solution to the following wave equation
\begin{equation}
\begin{cases}
   \bg r_\sigma =-(1-\tchi)\bg e^{i \sigma \phi} a_\sigma & \textnormal{in} \: (\tau_0,\tau_1) \times \Sigma,\\
     r_{\sigma|t=0}= \p_t r_{\sigma|t=0}=0  .
    \end{cases}
\end{equation}
We define $w=(1-\tchi) e^{i \sigma \phi}a_\sigma+ r_\sigma$ with $a$ the Gaussian beam constructed $e^{i \sigma \phi}a_\sigma$ with the choices of $a_j$ described above. Then $w$ satisfies  
\begin{equation}
\begin{cases}
   \bg w =[\tchi,\bg] e^{i \sigma \phi} a_\sigma & \textnormal{in} \: (\tau_0,\tau_1) \times  \Sigma,\\
   w_{|t=0}= \p_t w_{|t=0}=0 .
    \end{cases}
\end{equation}
The construction of $a_j$ implies that they are supported close to $\gamma$, hence $a$ is supported close to $\gamma$ as well. For small $\eta^\prime$ we have that $[\tchi,\bg] e^{i \sigma \phi} a_\sigma$ is supported close to $\gamma (a+\eta) \notin K$. Then, $[\tchi,\bg] e^{i \sigma \phi} a_\sigma$ is a smooth function supported away from $K$ and independent of $Q$. This implies that $\rr$ determines $w_{| \M \backslash K}$. Remark that strictly speaking, the source to solution map $\rr$ as defined by Proposition~\ref{prop_wave_lorentz} corresponds to solutions of the wave equation on the whole $\M = \R \times \Sigma$. However, one can uniquely extend the solution $w$ to a solution $\Tilde{w}$ of 
\begin{equation}
\begin{cases}
   \bg \Tilde{w} = \Tilde{f} & \textnormal{in} \: \R \times  \Sigma,\\
   \tilde{w}_{|t=0}= \p_t \tilde{w}_{|t=0}=0  ,
    \end{cases}
\end{equation}
with $\Tilde{f}$ the extension by $0$ of $[\tchi,\bg] e^{i \sigma \phi} a_\sigma$ to $\R \times \Sigma$ and $\tilde{w}$ satisfying $\tilde{w}(t)= w(t), \: t \in (\tau_0,\tau_1) $ and $\tilde{w}(\tau_1)=w(\tau_1)$, $\p_t \tilde{w}(\tau_1)=\p_t w(\tau_1)$. We then see that $\rr$ determines $\tilde{w}_{| \M \backslash K}$ and hence in particular $w_{| \M \backslash K}$.

To conclude, notice that the choice of $\tau_0, \tau_1$ implies that $\gamma(a), \gamma(b) \notin [\tau_0, \tau_1] \times \Sigma$. Hence, we can apply the Gaussian beam decay Lemma~\ref{lem_gauss_beam_decay} and conclude as in Lemma~\ref{lem_lorent_Rdetermines_sol}. We denote $\bd=b-\eta$. Sending $\sigma \to \infty$ and using that $\gamma(\bd) \notin K$ we obtain that $\rr$ determines $a_1(\gamma(\bd))$. Finally, we remark that in the Fermi coordinates $\gamma(\bd)=(\bd,0)$ and therefore $a_{1,j}(\gamma(\bd))=a_{1,j}(
\bd,0)=0$ for $j \geq 1$, which implies $a_{1}(\gamma(\bd))=a_{1,0}(\gamma(\bd))$. The lemma follows by sending $\eta \to 0$.
\end{proof}

The last step to prove the analogue of Proposition~\ref{prop_minkow_reduction} in this general geometric setting is to establish the link between the solutions of~\eqref{eq_for_a0_lorentz}, \eqref{eq_for_a1} and the parallel transport along $\gamma$ for the connection $-A$. As in the Minkowski case we write $P_{A, \gamma}(s)$ for the normalized fundamental matrix solution to the system of ODEs
\begin{equation}
     \p_s P_{A, \gamma}(s)- (A \cdot  \dot \gamma )P_{A, \gamma}(s)=0, \quad s \in I,\quad P_{A, \gamma}(\gamma(a))=\textnormal{Id}.
\end{equation}
We then define $r=r(s)$ by $r(s)=-\int_{\gamma(\alpha)}^{s}\ct(s^\prime)ds^\prime$. The important point is that $r$ is the solution to the ODE
\begin{equation}
\label{prop_of_r}
    \frac{d}{ds}r+ \ct =0 ,\quad s \in I, \quad  r(\gamma(a))=0.
\end{equation}
A straightforward calculation shows that $e^{r(s)} P_{A, \gamma}(s)a_{0,0}(\gamma(a))$ solves the system of ODE's~\eqref{eq_for_a0_lorentz} and therefore
\begin{equation}
\label{expression_for_a00}
    a_{0,0}(s)=e^{r(s)} P_{A, \gamma}(s)a_{0,0}(\gamma(a)).
\end{equation}
With this observation we can now prove the reduction theorem in the general case.
\begin{proof}[Proof of Theorem~\ref{thm_reduction_lorentz}]
  Let $b_1$ be the unique solution to the system of linear ODE's   
    \begin{equation}
    \p_s b_1+\ct b_1- (A \cdot \dot\gamma )b_1=-\frac{i }{2}\Box_{\g} a_{0}-i A \cdot \nabla^{\g} a_{0},\quad s \in I, \quad b_1(\gamma(a))=0.
\end{equation}
This solution $b_1$ is independent of $Q$. As a consequence, combining this with Lemma~\ref{lem_lorent_Rdetermines_sol} we find that $\rr$ determines $\Tilde{c}_1 (\gamma(b))$ with $\Tilde{c}_1=-\frac{2}{i}(a_{1,0}-b_1)$. Observe that $\Tilde{c}_1$ solves
\begin{equation}
    \frac{d}{ds}+ \ct \Tilde{c}_1-(A \cdot \dot\gamma )\Tilde{c}_1=Qa_{0,0}, \quad s\in I, \quad 
    \Tilde{c}_1(\gamma(a))=0,
\end{equation}
where we used the fact on $\gamma$ one has $Qa_0=Qa_{0,0}$. We consider as well $c_1$ solution to
\begin{equation}
    \frac{d}{ds}c_1-(A \cdot \dot\gamma )c_1=e^{-r}Qa_{0,0}, \quad s \in I, \quad c_1(\gamma(a))=0.
\end{equation}
Then using~\eqref{prop_of_r} we immediately get that $\Tilde{c}_1=e^rc_1$, therefore $\rr$ uniquely determines $c_1(\gamma(b))$. But using~\eqref{expression_for_a00} we see that in fact $c_1$ solves 
\begin{equation}
    \frac{d}{ds}c_1-(A \cdot \dot\gamma )c_1= Q P_{A, \gamma}(s)a_{0,0}(\gamma(a)), \quad s \in I, c_1(\gamma(a))=0.
\end{equation}
The solution of this last ODE can be expressed similarly to the Minkwoski case as
\begin{equation}
    c_1(s)=P_{A,\gamma}(s)\int_{\gamma(a)}^{s} P_{A,\gamma}(s^\prime)^{-1} Q(\gamma(s^\prime))P_{A,\gamma}(s^\prime)ds^\prime a_{0,0}(\gamma(a)).
\end{equation}
Since $\rr$ determines $c_1(\gamma(b))$ we get finally that it determines as well
\begin{equation}
    \int_{\gamma(a)}^{\gamma(b)} P_{A,\gamma}(s)^{-1} Q(\gamma(s))P_{A,\gamma}(s)ds,
\end{equation}
for any null geodesic $\gamma : [a-\delta, b+\delta] \to \M$ with $\gamma(a) , \gamma(b) \notin K$. This proves the stated result.  
\end{proof}

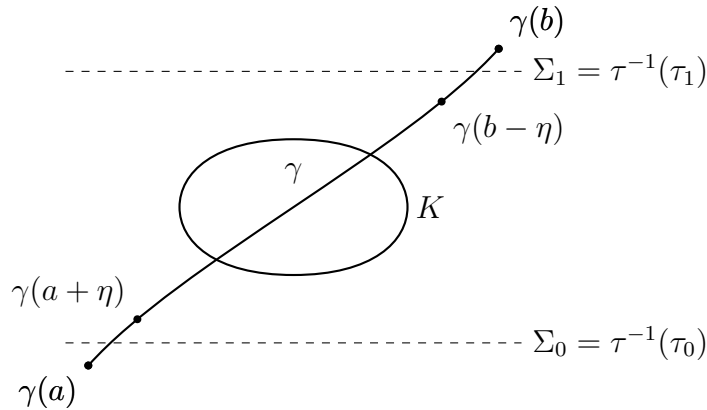
\begin{figure}[htbp]
  \centering
  \begin{tikzpicture}[scale=1.5]
    \draw[dashed] (-2, -1.2) -- (2, -1.2) node[right] {$\Sigma_0= \tau^{-1}(\tau_0)$};
    \draw[dashed] (-2, 1.2) -- (2, 1.2) node[right] {$\Sigma_1=\tau^{-1}(\tau_1)$};

   \draw[thick] 
  (-1.8, -1.4) .. controls (-1, -0.5) and (1, 0.5) .. (1.8, 1.4);

\fill (-1.8, -1.4) circle (1pt) node[below left] {$\gamma(a)$};
\fill (1.8, 1.4) circle (1pt) node[above right] {$\gamma(b)$};

    \node at (0,0.3) {$\gamma$};

    \draw[thick]
      (0,0.6)
        .. controls (0.7,0.6) and (1,0.3) .. (1,0)
        .. controls (1,-0.3) and (0.7,-0.6) .. (0,-0.6)
        .. controls (-0.7,-0.6) and (-1,-0.3) .. (-1,0)
        .. controls (-1,0.3) and (-0.7,0.6) .. (0,0.6);

    \node at (1.2, 0) {$K$};

\fill (-1.8, -1.4) circle (1pt) node[below left] {$\gamma(a)$};
\fill (1.8, 1.4) circle (1pt) node[above right] {$\gamma(b)$};

\draw[thick, decorate, 
    decoration={
        markings,
        mark=at position 0.13 with {
            \fill (0,0) circle (1.5pt) node[above left] {$\gamma(a+\eta)$};
        },
        mark=at position 0.85 with {
            \fill (0,0) circle (1.5pt) node[below right] {$\gamma(b-\eta)$};
        },
    }
] 
(-1.8, -1.4) .. controls (-1, -0.5) and (1, 0.5) .. (1.8, 1.4);

  \end{tikzpicture}
  \caption{The null geodesic $\gamma$ and the time slices giving the compact interval in time where the Gaussian beam decay lemma is applied.}
  \label{fig:gamma-curve}
\end{figure}

\section{Stationary geometry}\label{sec-stat}

We now assume that the globally hyperbolic $(n+1)-$manifold $\mathcal{M}$ is stationary, i.e., there is a complete timelike Killing vector field $X$ on $\mathcal{M}$. Our unknown potential $Q$ is supported in some compact $K \subset \mathcal{M}$. In order to carry out a reduction to a Riemannian ray transform problem, we will construct a Lorentzian submanifold $\olm \subset \mathcal{M}$ with boundary and $\dim (\olm) = n+1$.

Since $\mathcal{M}$ is globally hyperbolic, Theorem~\ref{prop_of_globhyper} guarantees the existence of a smooth spacelike Cauchy hypersurface $\Sigma \subset \mathcal{M}$. Then $\mathcal{M}$ is isometric to a standard stationary manifold \cite[Lemma 3.3]{Javaloyes:2008zz}, which means that the metric can be expressed in coordinates $(t,x)$ as 
\begin{equation}\label{standard-stationary}
    \overline{g} = c(x)^2\left(- dt^2 + 2\tilde{\omega}(x) dt + \tilde{g}(x)\right),
\end{equation}
where $t$ is a global time coordinate associated with the flow of $X$, and the 1-form $\tilde{\omega}$ and the Riemannian metric $\tilde{g}$ are independent of time $t$. Moreover, $X = \pt_t$ in these coordinates. We review the proof of this result in Appendix \ref{app-stat}. With definitions $g = \tilde{g} + \tilde{\omega}\otimes \tilde{\omega}$ and $\omega = - \tilde{\omega}$, we can complete the square so that the metric becomes
\begin{equation}\label{gstat1}
\overline{g} = c^2(-(dt + \omega)^2 + g).
\end{equation}
Here the new 1-form $\omega$ and the Riemannian metric $g$ are also independent of time $t$.

Let $\pi : \R \times \Sigma \to \Sigma$ be the natural projection. Then the image $\pi(K)$ is compact in $\Sigma$. 
Now choose a compact Riemannian submanifold $M \subset \Sigma$ with smooth boundary such that $\pi(K) \subset \text{Int} (M)$. This choice is justified provided that, e.g., there is a minimizing inextendible null geodesic\footnote{Recall that a null geodesic $\gamma: I \to \mathcal{M}$ is minimizing if the Lorentzian distance function restricted to $\gamma(I)\times \gamma(I)$ vanishes. In other words, there is no null cut point along $\gamma$.}, and with the help of Lemma \ref{exit-lemma}. The rationale here is that while Theorem~\ref{thm_reduction_lorentz} works for any $K$ compact, the reduction to Riemannian ray tomography problem may still fail if, for instance, the Cauchy hypersurface $\Sigma$ is compact and $K = I \times \Sigma$ with $I \subset \R$ some closed interval. This could happen, e.g., in the Einstein static universe $\R \times S^n$ endowed with the natural Lorentzian product metric. On the other hand, when $\Sigma$ is noncompact, the problem does not arise; $\pi(K)$ can be contained in an open set with compact closure in $\Sigma$.

Next, we set $\olm = \R \times M$ and point out that $K$ is a compact subset of $\text{Int} ( \olm )$. Now $\overline{M} \cong \R \times M$ is a standard stationary Lorentzian $(n+1)-$manifold with boundary, a `spacetime cylinder,' whose metric is given by \eqref{gstat1}. Injectivity of the light ray transform is conformally invariant (see Appendix \ref{app-conformal} below) so we can rescale \eqref{gstat1} and consider instead an ultrastationary manifold with the metric 
\begin{equation}\label{gstat2}
\overline{g} = -(dt + \omega)^2 + g.
\end{equation}
We write a point $p \in \overline{M}$ as $p=(t,x)$ where $x$ is some local coordinate system on $M$. Any null tangent vector $\overline{v} = (v^0, v) \in T_p\overline{M}, v \in T_x M,$ satisfies
\begin{equation}
    (v^0 + \omega_x(v))^2 = \norm{v}_g^2 \ ,
\end{equation}
where $\norm{\cdot}_g$ is the norm induced by $g$. Thus, given a future-directed null geodesic $\gamma(s) = (t(s),x(s))$, it holds that 
\begin{equation}
    \dot{t}=\norm{\dot{x}}_g - \omega_x(\dot{x}).
\end{equation}
It is known that null geodesics of \eqref{gstat2} can be fully described in terms of magnetic geodesics of the magnetic system $(M,g,\omega)$ that satisfy the Lorentz force equation \eqref{eq_lorentz_force}. (For a detailed exposition of this story, see e.g. \cite{oksanen2025interplay}.) Magnetic geodesics have constant speed; we pick here unit speed normalization $\norm{\dot{x}}_g = 1$ so that the corresponding null geodesics are
\[
\gamma(s) = (t_0 + s - \int_0^s \omega_{x(\sigma)}(\dot{x}(\sigma)) d\sigma, x(s)),
\]
where $t_0$ is the initial time for the geodesic.
This corresponds to a choice of affine parametrization for null geodesics.

We denote the flow of unit speed magnetic geodesics by $\varphi_s$, and the correspondingly normalized null geodesic flow by $\phi_s$. The vector bundle of normalized null directions is the subset of $T\overline{M}$ given by 
\[
    \{ (z,\overline{v}) \in T\overline{M}: z = (t,x) \in \overline{M}, \overline{v}=(v^0, v) \in T_z\overline{M} , \norm{v}_g = 1, v^0 = 1 - \omega_x(v)\}.
\]
Since the timelike component $v^0$ is fully determined by the magnetic potential and the unit spacelike direction, we denote points in the bundle by $(t,x,v)$ and the bundle itself by $\R \times SM$ where $SM$ is the unit sphere bundle of $M$. The flow $\phi_s$ is then a map $\R\times SM \to \R \times SM$.

Let $G$ and $X$ be the infinitesimal generators of $\varphi_s$ and $\phi_s$, respectively. The geodesic vector field $X$ can be written explicitly as
\begin{equation}\label{X-local}
X(t,x,v) = (1-\omega_x(v))\pt_t + G(x,v).
\end{equation}

 The influx boundary $\pt_+ SM$ is defined as the set of inwards-pointing unit tangent vectors at the boundary $\pt M$ and, similarly, the outflux boundary $\pt_- SM$ is formed of the outwards-pointing unit tangent vectors at the boundary, so that $\pt SM = \pt_- SM \cup \pt_+ SM$. The intersection $\pt_0 SM = \pt_+ SM \cap \pt_- SM = S(\pt M)$ is called the glancing region \cite[Ch. 3]{paternain2023geometric}.

 We assume that $\varphi_s$ is non-trapping with exit time function $\kappa: SM \to [0,\infty)$. Analogously, the enter time is the function $\sigma: SM \to (-\infty,0]$ such that the magnetic geodesic $\gamma_{x,v}: [\sigma(x,v),\taux] \to M$ is maximal. We also assume that $M$ is strictly magnetic convex, which means that
 \begin{equation}\label{eq_magnetic_convex}
 \Pi_{x}(v,v) > g_{x}(F_{x}(v),\nu(x))
 \end{equation}
 for all $x\in \partial M$ and $v\in T_{x}\partial M$. Here $\Pi$ denotes the (scalar) second fundamental form of $\partial M$ and $\nu$ the inward unit normal. Strict magnetic convexity of the boundary guarantees that $\kappa$ is smooth in $SM \setminus \pt_0 SM$, which follows from the implicit function theorem as in \cite[Lem. 3.2.3]{paternain2023geometric}.

\section{Non-Abelian light ray transform}
\label{sec_non_abelian_light}

\subsection{General connection}
Before specializing to the time independent case, we first define the non-Abelian light ray transform and derive the corresponding transport equation for a general connection.
Consider the trivial bundle $E=\overline{M}\times \C^N$ with a smooth unitary connection one-form $\ola$. At first, we allow $\ola$ to depend on time, and we write the connection as
\[
\overline{A}(t,x) = \Phi(t,x) dt + \tilde{A}(t,x),
\]
    where $\Phi$ is a smooth matrix-valued function on $\olm$ (called a Higgs field in e.g. \cite{paternain2023geometric}), and $\tilde{A}$ is a time dependent connection one-form on $M \times \C^N$. The connection acting in null directions splits into a time dependent connection and a Higgs field on the Riemannian base manifold as
\[
\ola(\dot{\gamma}_{t,x,v}(0)) = \Phi(t,x)(1- \omega(v)) + \tilde{A}(v) = \Phi(t,x) + B(v),
\]
where $B = \tilde{A} - \Phi \, \omega$ is a new connection one-form on $M$.

Let $Q$ be a smooth function taking values in $\C^{N\times N}$.
Given a maximal null geodesic $\gamma: [0,l] \to \overline{M}$, we define the light ray transform of $Q \in \Cinf(\olm, \C^{N\times N})$ with $\supp Q \subset \text{Int}(\olm)$ as
\begin{equation}
    \LR_{\ola} Q(\gamma) = \int_0^l P_{\overline{A},\gamma}(s)\pminus{1} Q(\gamma(s)) P_{\overline{A},\gamma}(s) ds,
\end{equation}
where $P_{\overline{A},\gamma}(s)$ is the parallel transport map for the connection $\overline{A}$ in Theorem~\ref{thm_reduction_lorentz}. Recall that parallel transport of a vector $u_0 \in \C^N$ along $\gamma$ is the unique solution of the initial value problem
\[
    \frac{d}{ds}u(s) + \overline{A}(\dot{\gamma}(s))u(s) = 0, \quad    u(0) = u_0,
\]
and $P_{\overline{A},\gamma}(s)$ is the fundamental matrix solution to this equation.

We assume that $\LR_{\ola} Q = 0$ and aim to show that this entails $Q = 0$. This problem is more effectively studied at the level of tangent bundle where the integral geometry problem reduces to a transport problem.
In this picture, the connection $\ola$ is mapped to a \emph{matrix attenuation} $\A$ on $\R \times SM$, given by $\A = (\pi^*_{\olm} \ola)(X)$, where $\pi_{\olm}: T\olm \to \olm$ is the canonical projection. In local coordinates, this means that
\[
\A(t,x,v) = \ola(\dot{\gamma}_{t,x,v}(0)) = \Phi(t,x) + B(v)
\]
for the null geodesic $\gamma_{t,x,v}$ with initial data $(t,x,v)\in \R\times SM$.
Similarly, the parallel transport map $P_{\overline{A},\gamma}$ becomes a map
$P_{\A}:\mathcal{O}\subset\R^2 \times SM \to \C^{N\times N}$ by setting
\[
P_{\A} (s,t,x,v) = P_{\ola,\gamma_{t,x,v}}(s),
\]
where $\gamma_{t,x,v}$ is a null geodesic as above, and the domain $\mathcal{O}$ is such that for a fixed $(t,x,v)\in \R\times SM$, the map $P_\A(\cdot,t,x,v)$ is defined on the interval $[\sigma(x,v),\taux]$.
The parallel transport map $P_{\A}$ satisfies the ODE
\begin{equation}\label{eq:parallel}
\frac{d}{ds}P_\A (s,t,x,v) + \A (\phi_s(t,x,v)) P_\A (s,t,x,v) = 0 ,\quad P_\A(0,t,x,v) = \text{Id}.
\end{equation}
Furthermore, it has the cocycle property
\begin{equation}\label{cocycle}
    P_\A (s + s', t,x,v) = P_\A(s',\phi_s(t,x,v)) P_\A (s,t,x,v), \quad \forall \, \sigma(x,v) \leq s + s' \leq \taux
\end{equation}
which follows from uniqueness for ODEs and the flow composition rule.

We also promote $Q$ to a degree 0 function in $\Cinfo(\R \times SM, \C^{N\times N})$.
Define a function $U^Q$ on $\R \times SM$ by
\begin{equation}\label{U-def}
    U^Q(t,x,v) = \int_0^\taux P_\A(s,t,x,v)\pminus{1} Q(\phi_s(t,x,v)) P_\A(s,t,x,v) ds.
\end{equation}
Since $P_{\A}$ and $Q$ are smooth everywhere, $\kappa$ is smooth in $SM\setminus \pt_0 SM$ and continuous in $SM$, and $\supp Q \subset \text{Int} ( \olm )$, also $U^Q$ has to be smooth. Moreover, $U^Q$ is compactly supported, as $Q$ is.
We have $U^Q \rvert_{\R\times \pt_+ SM} = \LR_{\ola} Q$ and $U^Q\rvert_{\R\times\pt_- SM} = 0$. Observe that 
\begin{align}
    &U^Q(\phi_\sigma(t,x,v)) \nn \\ 
    = &\int_0^{\kappa(\varphi_\sigma(x,v))} P_\A(s,\phi_\sigma(t,x,v))\pminus{1} Q(\phi_{\sigma+s}(t,x,v)) P_\A(s,\phi_\sigma(t,x,v)) ds \nn \\
    = &P_\A(\sigma, t,x,v) \left(\int_0^{\taux-\sigma} P_\A(\sigma + s,t,x,v)\pminus{1} Q(\phi_{\sigma+s}(t,x,v)) P_\A(\sigma + s,t,x,v) ds \right)\\
    & \quad \cdot P_A(\sigma, t,x,v)\pminus{1},
\end{align}
where we used the cocycle property \eqref{cocycle}.
Using the definition of a geodesic vector field, we then get
\begin{align}
    XU^Q(t,x,v) &= \frac{d}{d\sigma}\Big\rvert_{\sigma=0}U^Q(\phi_\sigma(t,x,v)) \nn \\
    &= \frac{d}{d\sigma}\Big\rvert_{\sigma=0} P_\A(\sigma, t,x,v) U^Q(t,x,v) + U^Q(t,x,v) \frac{d}{d\sigma}\Big\rvert_{\sigma=0} P_\A(\sigma, t,x,v)\pminus{1} \nn \\
    &\quad + \frac{d}{d\sigma}\Big\rvert_{\sigma=0}\int_0^{\taux-\sigma} P_\A(\sigma + s,t,x,v)\pminus{1} Q(\phi_{\sigma+s}(t,x,v)) P_\A(\sigma + s,t,x,v) ds \nn \\
    &= -\A(t,x,v) U^Q(t,x,v) + U^Q(t,x,v)\left(-\frac{d}{d\sigma}\Big\rvert_{\sigma=0} P_\A (\sigma,t,x,v) \right) \nn \\
    &\quad -  P_\A(\taux,t,x,v)\pminus{1} Q(\phi_{\taux}(t,x,v)) P_\A(\taux,t,x,v) \nn \\
    &\quad + \int_0^{\taux}\frac{d}{d\sigma}\Big\rvert_{\sigma=0} \left( P_\A(\sigma + s,t,x,v)\pminus{1} Q(\phi_{\sigma+s}(t,x,v)) P_\A(\sigma + s,t,x,v) \right) ds \nn \\
    &= -\A(t,x,v) U^Q(t,x,v) + U^Q(t,x,v) \A(t,x,v)  \nn \\
    &\quad -  P_\A(\taux,t,x,v)\pminus{1} Q(\phi_{\taux}(t,x,v)) P_\A(\taux,t,x,v) \nn \\
    &\quad + \int_0^{\taux}\frac{d}{ds} \left( P_\A( s,t,x,v)\pminus{1} Q(\phi_{s}(t,x,v)) P_\A(s,t,x,v) \right) ds \nn \\
    &= -\A(t,x,v) U^Q(t,x,v) + U^Q(t,x,v) \A(t,x,v) - Q(t,x,v).
\end{align}
On the third line here we used the parallel transport equation \eqref{eq:parallel} and the Leibniz integral rule, and on the fourth line we made a change of variables. From this we obtain that $U^Q$ is the unique solution to the transport problem
\begin{equation}\label{transport1}
    XU + [\A,U] = -Q, \quad U\rvert_{\R \times \pt_- SM} = 0.
\end{equation}
Injectivity of $\LR_{\ola}$ can be characterized in terms of the above transport problem.
\begin{lemma}\label{lemma-LR-charac}
    $\LR_{\ola} Q = 0$ implies $Q=0$ if and only if the unique smooth solution to 
    \[
    XU + [\A, U] = -Q, \quad U\rvert_{\R\times\pt SM} = 0
    \]
    is $U=0$.
\end{lemma}
\begin{proof}
    `$\Rightarrow$' Given a $Q$ as above, let $U$ solve the above transport problem. $U\rvert_{\pt_+ SM} = 0$ means that $\LR_{\ola} Q= 0$, which implies that $Q=0$. But $U$ is the unique solution to \eqref{transport1} and thus given by \eqref{U-def}, so $U$ vanishes when $Q=0$.

    `$\Leftarrow$' Suppose $\LR_{\ola} Q = 0$. Then $U^Q$ given in \eqref{U-def} is the unique solution to the transport problem in the lemma, which vanishes by assumption. Then the transport equation immediately gives $Q=0$.
\end{proof}
In the time dependent case proving injectivity of $\LR_{\ola}$ appears rather intractable; in the time independent case much more can be said.

\subsection{Time independent connection}
Suppose now that the connection is time independent, so that
\[
\overline{A}(x) = \Phi(x) dt + \tilde{A}(x),
\]
where $\Phi$ is now a smooth matrix-valued function on $M$, and $\tilde{A}(x)$ is a connection one-form on $M \times \C^N$. The connection acting in null directions now splits as
\[
\ola(\dot{\gamma}_{t,x,v}(0)) = \Phi(x) + B(v),
\]
where $B$ is a time independent connection one-form on $M$. Thus, parallel transport along light rays with respect to $\ola$ in spacetime projects to parallel transport along magnetic geodesics with respect to $(B,\Phi)$ on time slices. In the tangent bundle picture we now write
\[
\A(x,v) = \Phi(x) + B(v),
\]
for all $(x,v)\in SM$.

Time-independence opens the door for Fourier slicing techniques. Using the explicit form \eqref{X-local} for the null geodesic vector field $X$, \eqref{transport1} becomes
\begin{equation}
    GU + (1-\Omega)\pt_t U + [\A,U] = -Q, \quad U\rvert_{\R\times \pt SM} = 0
\end{equation}
where $\Omega(x,v)= \omega_x(v)$. We Fourier transform this with respect to $t$ to get
\begin{equation}\label{transport2}
    G \hat{U} + i\tau(1-\Omega)\hat{U} + [\A, \hat{U}] = - \hat{Q}, \quad \hat{U}\rvert_{\R\times \pt SM} = 0.
\end{equation}
Observe that the Fourier transforms exist and are real analytic since $Q$ and $U$ have compact support. This allows us to prove
\begin{lemma}\label{lemma-reduction}
    Given a function $Q\in \Cinfo(\olm,\C^{N\times N})$ with $\supp Q \subset \textnormal{Int}(\olm)$, suppose that the solution $W \in \Cinf(SM,\C^{N\times N})$ to the transport problem
    \[
    G W + [\A,W] = - \hat{Q}, \quad W \rvert_{\pt SM} = 0
    \]
    is $W=0$. Then there exists a smooth solution $U$ to the transport problem \eqref{transport1} only if $Q = 0$.
\end{lemma}
\begin{proof}
Suppose there exists a smooth solution $U$ to the transport problem \eqref{transport1}. Then for each $\tau \in \R$ there exists a smooth $\hat{U}(\tau) \in \Cinf(SM, \C^{N\times N})$ that solves \eqref{transport2}.
Differentiate \eqref{transport2} now $k$ times with respect to $\tau$ and evaluate at $\tau=0$, and we have
\[
G\pt_\tau^k \hat{U}(0) + i(k-1)(1-\Omega) \pt_\tau^{k-1}\hat{U}(0) + [\A, \pt_\tau^k \hat{U}(0)] = - \pt_\tau^k \hat{Q}(0), \quad \pt_\tau^k\hat{U}(0)\rvert_{\pt SM} = 0.
\]
We establish the basic step by noting that
\[
G \hat{U}(0) + [\A, \hat{U}(0)] = - \hat{Q}(0), \quad \hat{U}(0)\rvert_{\pt SM} = 0
\]
implies that $\hat{U}(0)=0$ and this in turn gives $\hat{Q}(0) = 0$. For the induction step, assume that $\pt_\tau^{k-1}\hat{U}(0) = 0$. Then we get
\[
G \pt_\tau^k\hat{U}(0) + [\A,\pt_\tau^k \hat{U}(0)] = - \pt_\tau^k \hat{Q}(0), \quad \pt_\tau^k\hat{U}(0)\rvert_{\pt SM} = 0
\]
from which it follows that $\pt_\tau^k\hat{U}(0) = \pt_\tau^k\hat{Q}(0)=0$. This completes the induction, and we find that
\[
\hat{Q}(\tau) = \sum_{k=0}^\infty \frac{\pt_\tau^k \hat{Q}(0)}{k!} \tau^k = 0 \quad \forall \tau \in \R.
\]
Hence, we conclude that $Q = 0$.
\end{proof}

Recall that we defined a non-Abelian magnetic X-ray transform $I_\A$ of $V \in \Cinf(SM,\C^{N\times N})$ with $\supp V \subset \textnormal{Int}(SM)$ by
\[
I_\A V(x,v) = \int_0^\taux P_{\A}(s,x,v)\pminus{1}V(\varphi_s(x,v)) P_{\A}(s,x,v) ds, \quad \forall (x,v) \in \pt_+ SM,
\]
and, as before, we define a matrix-valued function
\begin{equation}\label{UW-def}
   W^V(x,v) = \int_0^\taux P_{\A}(s,x,v)\pminus{1}V(\varphi_s(x,v)) P_{\A}(s,x,v) ds, \quad \forall (x,v) \in SM.
\end{equation}
Here $P_{\A}(s,x,v)$ is the smooth map given by the parallel transport along magnetic geodesics on $M$, which coincides with $P_{\A}(t,s,x,v)$ in the time independent case. A similar calculation as above gives that $W^V$ is the unique solution to the transport equation
\begin{equation}\label{transport3}
G W + [\A, W] = - V, \quad W\rvert_{\pt_- SM} = 0.
\end{equation}
Just like $\LR_{\ola}$, injectivity of $I_\A$ can be characterized by the properties of the corresponding transport equation. 
\begin{lemma}\label{lemma-IA-charac}
    $I_\A V = 0$ implies $V=0$ if and only if the unique smooth solution to 
    \[
    GW + [\A, W] = -V, \quad W\rvert_{\pt SM} = 0
    \]
    is $W=0$.
\end{lemma}
\noindent The proof of this lemma is completely analogous to that of Lemma \ref{lemma-LR-charac}.

Combining Lemmata~\ref{lemma-LR-charac}, \ref{lemma-reduction}, and \ref{lemma-IA-charac}, we arrive at our main theorem.
\begin{theorem}
\label{thm_sufficient_condition_for_light_ray}
    If $I_\A$ is injective for functions in $\{V \in \Cinfo(M, \C^{N\times N}): \supp V \subset \intr{M}  \}$, then $\LR_{\ola}$ is injective for functions in $\{Q \in \Cinfo(\R\times M, \C^{N\times N}): \supp Q \subset \R\times \intr{M}  \}$.
\end{theorem}
\begin{proof}[Proof of Corollary~\ref{cor_A_nonzero}]
    We can understand the commutator in~\eqref{transport1} as a new matrix attenuation living in $\Cinf(\R\times SM, \C^{N^2 \times N^2})$ that corresponds to a connection on the endomorphism bundle $\text{End}(E) = \olm \times \C^{N\times N}$ (see e.g.\cite{tetlow2022recovery}). Indeed, defining the linear map $\B V = [\A, V]$, we see from
\[
\B V (t,x,v) = [\Phi(t,x), V(t,x,v)] + [B_i(t,x), V(t,x,v) ] v^i
\]
that $\B$ corresponds to a matrix attenuation given by a new Higgs field and a connection defined by the above commutators, i.e.
\begin{align}
    \tilde{\Phi}V = [\Phi, V], \quad \tilde{B}_i V = [B_i, V]
\end{align}
For a unitary connection $\ola$, the matrices $\Phi$ and $B_i$ are skew-Hermitian. The natural inner product on the endomorphism bundle is given by $\innerproduct{X}{Y}_{\text{End}(E)} = \Tr(X^\dagger Y)$, where $\dagger$ denotes the conjugate transpose. Thus, for a skew-Hermitian $\Phi$, we have
\begin{align}
\innerproduct{\tilde{\Phi} X}{Y}_{\text{End}(E)} &= \Tr([\Phi,X]^\dagger Y) = \Tr(X^\dagger \Phi^\dagger Y - \Phi^\dagger X^\dagger Y) = \Tr(X^\dagger(- \Phi Y + Y \Phi))\\
&= - \innerproduct{X}{\tilde{\Phi}Y}_{\text{End}(E)},    
\end{align}
where we used the cyclicity of trace. The same calculation can be applied to $\tilde{B}_i$. Hence, $\tilde{\Phi}$ and $\tilde{B}_i$ are skew-Hermitian linear maps with respect to $\innerproduct{\cdot}{\cdot}_{\text{End}(E)}$, with $N^2 \times N^2$ matrices of the linear maps acting on vectors in $\C^{N^2} \cong \C^{N\times N}$. The linear map $\B$ can therefore be understood as a $\C^{N^2\times N^2}$ matrix attenuation so that the transport problem \eqref{transport1} can be written in the standard form
\[
XU + \B U = -Q, \quad U\rvert_{\R \times \pt SM} = 0.
\]
Similarly, the magnetic transport equation becomes
\[
GU + \B U = - \hat{Q}, U\rvert_{\pt SM} = 0.
\]
The injectivity of $I_\A$ for functions is thus equivalent with the above transport problem having a vanishing solution.
\begin{enumerate}
    \item Suppose that $M$ is a simply connected surface having a strictly convex boundary with respect to $\varphi$, and that $\varphi$ has no conjugate points. Then \cite[Thm. 1.2]{Ains_13} implies that $I_\A$ is injective. Our main theorem then gives that the source-to-solution map $\mathcal R$ determines $Q$.
    \item The Lorentzian manifold being static means that $\omega = 0$. Then the non-Abelian magnetic X-ray transform coincides with the non-Abelian geodesic X-ray transform. Assuming that $\dim (M) \geq 3$, $\pt M$ is strictly convex, and $M$ admits a strictly convex foliation, \cite[Thm. 1.6]{paternain2019} implies that $I_\A$ is injective, from which we get again the conclusion.
\end{enumerate}
\end{proof}

\appendix

\section{Injectivity of the non-Abelian light ray transform is conformally invariant}\label{app-conformal}

Recall that a smooth curve $\alpha$ is called a pregeodesic provided that there exists a reparametrization $h$ such that $\alpha \circ h$ is a geodesic \cite{oneill1983}.

\begin{lemma}\label{conformal-null}
    Let $g, \Tilde{g}$ be conformally equivalent Lorentzian metrics on $\mathcal M$ with $\tilde{g} = c^2 g$ for some $c \in \Cinf(\mathcal M)$. Then we have
    \begin{equation}\label{pregeod}
        \widetilde{\na}_{\dot{\gamma}} \dot{\gamma} = 2 c\pminus{1}(dc \cdot \dot{\gamma})\dot{\gamma},
    \end{equation}
    for any null geodesic $\gamma$ of $g$. Furthermore, the null geodesics of $g$ are null pregeodesics of $\Tilde{g}$, and vice versa.
\end{lemma}
\begin{proof}
    The first claim follows from the conformal relation of Christoffel symbols
    \begin{equation}\label{Gamma-relation}
        \widetilde{\Gamma}^\rho\munud = \Gamma^\rho\munud + c\pminus{1} ((dc)\mud \delta^\rho\nud +  (dc)\nud \delta^\rho\mud - (\na c)^\rho g\munud),
    \end{equation}
    where $\widetilde{\Gamma}\rup\munud$ is a Christoffel symbol for $\tilde{g}$ and $\Gamma\rup\munud$ a Christoffel symbol for $g$. For the second claim, let $h$ be a reparametrization of $\gamma: [a,b]\to M$, denote $\Tilde{\gamma} = \gamma \circ h$, and let $\Tilde{s}$ be a parameter on the reparametrized curve such that $\Tilde{\gamma}(\Tilde{s}) = \gamma(s)$. Then we require that
    \begin{align}
        \widetilde{\na}_{\dot{\tilde{\gamma}}} \dot{\tilde{\gamma}} = \widetilde{\na}_{\tilde{\dot{\gamma}}} (h' \dot{\gamma}) = h'' \dot{\gamma} + (h')^2 \widetilde{\na}_{\dot{\gamma}} \dot{\gamma} = h'' \dot{\gamma} + 2 (h')^2 c\pminus{1}(dc \cdot \dot{\gamma}) \dot{\gamma} = 0.
    \end{align}
    Solving the first order ODE for $h'$, we get
    \begin{equation}\label{eq_hprime}
        h'(\tilde{s}) = \left( \frac{c(\tilde{\gamma}(h\pminus{1}(a)))}{c(\tilde{\gamma}(\tilde{s}))} \right)^2.
    \end{equation}
    Thus, there exists a reparametrization $h$ such that $\widetilde{\na}_{\dot{\tilde{\gamma}}} \dot{\tilde{\gamma}} = 0$. Moreover, the causal character of a curve is conformally invariant so we conclude that $\gamma$ is a null pregeodesic of $\tilde{g}$. The argument is similar for null geodesics of $\tilde{g}$, with $c$ replaced by $c\pminus{1}$.
\end{proof}

\begin{proposition}
    Let $\tilde{g}$ be a Lorentzian metric on $\mathcal M$ and suppose that the non-Abelian light ray transform is injective on $(\mathcal M,\tilde{g})$. Then the non-Abelian light ray transform is injective in the whole conformal class of $(\mathcal M,\tilde{g})$.
\end{proposition}
\begin{proof}
    Let $g$ be conformally equivalent to $\Tilde{g}$ and let $c \in \Cinf(M)$ be a function such that $\Tilde{g} = c^2 g$. 
    Denote the light ray transforms of $g$ and $\tilde{g}$ by $\LR_{\ola}$ and $\tilde{\LR}_{\ola}$, respectively, and suppose that $\tilde{\LR}_m$ is injective. Let $Q \in \Cinf(\mathcal M, \C^{N\times N})$ be such that $\LR _{\ola} Q = 0$, that is,
    \begin{equation}
        \int_0^l P_{\ola,\gamma} (s)\pminus 1 Q(\gamma(s)) P_{\ola, \gamma}(s) ds = 0,
    \end{equation}
    for all maximal null geodesics $\gamma$.
    Let now $\tilde \gamma = \gamma \circ h$ as in the previous lemma, and $P_{\ola, \tilde \gamma} = P_{\ola, \gamma}\circ h$. Observe then that
    \begin{equation}
        P_{\ola, \tilde \gamma}' + (\ola \cdot \dot {\tilde \gamma})P_{\ola, \tilde \gamma} = h' P_{\ola, \gamma}' + h' (\ola \cdot \gamma) P_{\ola, \gamma} = 0, \quad P_{\ola, \tilde \gamma}(h \pminus 1 (0)) = P_{\ola, \gamma}(0) = \text{id}. 
    \end{equation}
    Uniqueness for ODEs then implies that $P_{\ola, \tilde \gamma} = P_{\ola, \gamma}$. Therefore, we get
    \begin{align}
        \int_0^l P_{\ola,\gamma} (s)\pminus 1 Q(\gamma(s)) P_{\ola, \gamma}(s) ds &= \int_{h\pminus 1 (0)}^{h\pminus 1 (l)} P_{\ola, \tilde \gamma}(\tilde s)\pminus 1 Q(\tilde \gamma(\tilde s)) P_{\ola, \tilde \gamma}(\tilde s) h'(\tilde s) d\tilde s \\
        &= c(\gamma (0))^2 \int_{h\pminus 1 (0)}^{h\pminus 1 (l)} P_{\ola, \tilde \gamma}(\tilde s)\pminus 1 Q(\tilde \gamma(\tilde s)) P_{\ola, \tilde \gamma}(\tilde s) c(\tilde \gamma (\tilde s))\pminus 2 d\tilde s,
    \end{align}
    where we used \eqref{eq_hprime} from the previous lemma. From the injectivity of $\tilde \LR_{\ola}$ it then follows that $Q / c^2$ vanishes, which in turn implies that $Q = 0$.
\end{proof}

\section{Further details on stationary geometry}\label{app-stat}

\begin{lemma}
    Suppose $\mathcal{M}$ is a globally hyperbolic manifold that contains a complete timelike Killing vector field $X$. Then $\mathcal{M}$ is isometric to a standard stationary manifold.
\end{lemma}
\begin{proof}
Let $\tilde{\psi}: \mathcal{D} \to \mathcal{M}$ be the flow of $X$ where $\mathcal{D} = \R \times \mathcal{M}$ is the maximal flow domain. 
Then set $\mathcal{D}_\Sigma = \R \times \Sigma \subset \mathcal{D}$ and observe that $\mathcal{D}_\Sigma$ is a smooth hypersurface in $\mathcal{D}$, as $\Sigma$ is a smooth hypersurface in $\mathcal{M}$. Then the restriction $\psi = \tilde{\psi}\rvert_{\mathcal{D}_\Sigma}$ is smooth. Since $X$ is timelike it is nowhere tangent to the Cauchy hypersurface $\Sigma$. From the flowout theorem \cite[Thm. 9.20]{lee2002} it then follows that $\psi$ is an immersion between manifolds of the same dimension, and thus a local diffeomorphism. As in the proof of \cite[Prop. 14.31]{ONeill}, we see that $\psi$ is a bijection, and hence a full diffeomorphism. Then we observe that the Lorentzian manifold $(\R \times \Sigma, \psi^* \overline{g})$ is standard stationary. Indeed, given the coodinates $z=(t,x)$ on $\R \times \Sigma$, the coordinates $z' = z \circ \psi^{-1}$ on $\mathcal{M}$, and the associated coordinate basis vectors $\pt_i, \pt_i'$ with $i=0,...,n$, we calculate
\begin{align}
    \frac{\pt}{\pt t} \psi^* \overline{g}(\pt_i , \pt_j) &= \psi^* \overline{g}(\na_{\pt_t}^* \pt_i, \pt_j) + \psi^* \overline{g}(\pt_i, \na_{\pt_t}^*\pt_j) \nn \\
    &= \psi^* \overline{g}(\na^*_{\pt_i}\pt_t, \pt_j) + \psi^* \overline{g}(\pt_i, \na^*_{\pt_j}\pt_t) \nn \\
    &= \overline{g}(\na_{\psi_* \pt_i} \psi_* \pt_t, \psi_* \pt_j) + \overline{g}(\psi_* \pt_i, \na_{\psi_* \pt_j} \psi_* \pt_t),
\end{align}
where $\na$ is the Levi-Civita connection on $\mathcal{M}$, and $\na^* = \psi^*\na$ is the pullback connection, which coincides with the Levi-Civita connection of $\psi^* \overline{g}$. Then we note that $\psi_* \pt_t$ is in fact identical to the Killing field $X$, so that
\begin{equation}
    \frac{\pt}{\pt t} \psi^* \overline{g}(\pt_i , \pt_j) = \overline{g}(\na_{\pt_i'} X, \pt_j') + \overline{g}(\pt_i', \na_{\pt_j'} X) = 0,
\end{equation}
since $X$ is Killing. Consequently, the components of $\psi^*\overline{g}$ are independent of the time coordinate $t$, and the metric on $\R \times \Sigma$ can be written as in \eqref{standard-stationary}.
\end{proof}
The following lemma gives an \emph{intrinsic} sufficient condition for the complement of $\pi(K)$, defined in Section \ref{sec-stat}, to be nonempty and thus having an open neighborhood with compact closure.
\begin{lemma}\label{exit-lemma}
    Let $K\subset \mathcal{M}$ be compact and let $\tilde{\psi}$ be the flow of $X$. Then any minimizing future-inextendible null geodesic that starts in $\tilde{\psi}(\R \times K)$ eventually exits $\tilde{\psi}(\R \times K)$ never to return.
\end{lemma}
\begin{proof}
    We define a smooth map $\rho: \mathcal{M} \to \Sigma$ as $\rho = \pi_\Sigma \circ \psi^{-1}$ where $\pi_\Sigma: \R \times \Sigma \to \Sigma$ is the natural projection. Let $\gamma$ be such a null geodesic with a starting point $p=(t_0,x_0) \in \tilde{\psi}(\R \times K)$. We have $\tilde{\psi}(\R \times K) = \psi(\R \times \rho(K))$. Now fix $\epsilon > 0$ small and let $d_g(\cdot,\cdot)$ be the Riemannian distance on $\Sigma$ associated with the metric $g$. We consider a closed ball $\overline B$ with large enough radius such that $\rho(K) \subset \overline B$ and any piecewise smooth unit speed curve $\beta: [0,l] \to \Sigma$ with endpoints in $\rho(K)$ and of length $L(\beta) = l < d_g(\beta(0),\beta(l)) + \epsilon$ is contained in $\overline B$.
    
    Let then $L = \sup_{x,y \in \rho(K)} d_g(x,y) + \epsilon$ and $T > L (\sup_{x \in \overline B} \norm{\omega_x}_g +1)$. Let $q=(t_1,x_1) \in \R \times \rho(K)$ be such that $t_1 - t_0 \geq T$. Then there exists a piecewise smooth unit speed path $\beta: [0,l] \to \overline B$ from $x_0$ to $x_1$ with length $ l< d_g(x_0,x_1) + \epsilon$. Define now a piecewise smooth path $\alpha: [0,l] \to \mathcal{M}$ by
    \[
    \alpha(s) = (t(s),\beta(s)), \quad t(s) = (s t_1 + (l- s)t_0)/l.
    \]
    Now observe that
    \[
    \dot{t} = (t_1 - t_0)/l \geq T/l > L( \sup_{x \in \overline B}\norm{\omega_x}_g + 1)/l \geq \sup_{x \in \overline B}\norm{\omega_x}_g + 1
    \]
    so that the tangent vector of the path satisfies at any point
    \begin{align}
        c^{-2}\overline{g}(\dot{\alpha}(s),\dot{\alpha}(s)) &= - (\dot{t}(s) + \omega(\dot{\beta}(s)))^2 + \norm{\dot{\beta}(s)}_g^2 \nn \\
        &\leq - (\dot{t}(s) - \norm{\omega_{\beta(s)}}_g)^2 + 1 \nn \\
        &< - (\sup_{x \in \overline B}\norm{\omega_x}_g + 1 - \norm{\omega_\beta}_g)^2 + 1 \nn \\
        &\leq -1 + 1 = 0.
    \end{align}
    Thus, $\alpha$ is a piecewise smooth timelike curve from $p$ to $q$. Consequently, $q \in I^+(p)$ for all $q =(t,x) \in \R \times \rho(K)$ with $t \geq t_0 + T$. But then $\gamma$ has to exit $\psi(\R \times \rho(K))$ before the moment of time $t_0 + T$ and never return since it is minimizing, i.e. contained in $J^+(p)\setminus I^+(p).$
\end{proof}

\bibliographystyle{siam}
\bibliography{master}

\begin{thebibliography}{10}

\bibitem{aicha2015stability}
{\sc I.~B. Aicha}, {\em Stability estimate for hyperbolic inverse problem with time dependent coefficient}, arXiv preprint arXiv:1506.01935,  (2015).

\bibitem{Ains_13}
{\sc G.~Ainsworth}, {\em The attenuated magnetic ray transform on surfaces}, Inverse Probl. Imaging, 7 (2013), pp.~27--46.

\bibitem{alexakis2022}
{\sc S.~Alexakis, A.~Feizmohammadi, and L.~Oksanen}, {\em Lorentzian {Calder{\'o}n} problem under curvature bounds}, Invent. Math., 229 (2022), pp.~87--138.

\bibitem{ANIKONOVROMANOV+1997+487+490}
{\sc Y.~E. Anikonov and V.~G. Romanov}, {\em On uniqueness of determination of a form of first degree by its integrals along geodesics}, Journal of Inverse and Ill-posed Problems, 5 (1997), pp.~487--490.

\bibitem{babich1984complex}
{\sc V.~Babich and V.~Ulin}, {\em Complex space-time ray method and “quasiphotons”}, Journal of Soviet Mathematics, 24 (1984), pp.~269--273.

\bibitem{bar2007}
{\sc C.~B{\"a}r, N.~Ginoux, and F.~Pf{\"a}ffle}, {\em Wave equations on {Lorentzian} manifolds and quantization.}, ESI Lect. Math. Phys., Z{\"u}rich: European Mathematical Society Publishing House, 2007.

\bibitem{belishev1987approach}
{\sc M.~I. Belishev}, {\em An approach to multidimensional inverse problems for the wave equation}, in Doklady Akademii Nauk, vol.~297(3), Russian Academy of Sciences, 1987, pp.~524--527.

\bibitem{belishev1992boundary}
{\sc M.~I. Belishev and A.~P. Kachalov}, {\em Boundary controls and quasiphotons in a {R}iemannian manifold reconstruction problem via dynamical data}, Zapiski Nauchnykh Seminarov POMI, 203 (1992), pp.~21--50.

\bibitem{bernal2003}
{\sc A.~N. Bernal and M.~S{\'a}nchez}, {\em On smooth {Cauchy} hypersurfaces and {Geroch}'s splitting theorem}, Commun. Math. Phys., 243 (2003), pp.~461--470.

\bibitem{bernal2005}
\leavevmode\vrule height 2pt depth -1.6pt width 23pt, {\em Smoothness of time functions and the metric splitting of globally hyperbolic spacetimes}, Commun. Math. Phys., 257 (2005), pp.~43--50.

\bibitem{bukhgeim1981global}
{\sc A.~L. Bukhgeim and M.~V. Klibanov}, {\em Global uniqueness of a class of multidimensional inverse problems}, in Doklady Akademii Nauk, vol.~260(2), Russian Academy of Sciences, 1981, pp.~269--272.

\bibitem{dairbekov2007boundary}
{\sc N.~S. Dairbekov, G.~P. Paternain, P.~Stefanov, and G.~Uhlmann}, {\em The boundary rigidity problem in the presence of a magnetic field}, Advances in Mathematics, 216 (2007), pp.~535--609.

\bibitem{feizmohammadi2021b}
{\sc A.~Feizmohammadi, J.~Ilmavirta, Y.~Kian, and L.~Oksanen}, {\em Recovery of time-dependent coefficients from boundary data for hyperbolic equations}, J. Spectr. Theory, 11 (2021), pp.~1107--1143.

\bibitem{feizmohammadi2021light}
{\sc A.~Feizmohammadi, J.~Ilmavirta, and L.~Oksanen}, {\em The light ray transform in stationary and static lorentzian geometries}, The Journal of Geometric Analysis, 31 (2021), pp.~3656--3682.

\bibitem{feizmohammadi2022}
{\sc A.~Feizmohammadi and L.~Oksanen}, {\em Recovery of zeroth order coefficients in non-linear wave equations}, J. Inst. Math. Jussieu, 21 (2022), pp.~367--393.

\bibitem{hounnonkpe2019}
{\sc R.~A. Hounnonkpe and E.~Minguzzi}, {\em Globally hyperbolic spacetimes can be defined without the `causal' condition}, Classical Quantum Gravity, 36 (2019), p.~9.
\newblock Id/No 197001.

\bibitem{Javaloyes:2008zz}
{\sc M.~A. Javaloyes and M.~Sanchez}, {\em {A note on the existence of standard splittings for conformally stationary spacetimes}}, Class. Quant. Grav., 25 (2008), p.~168001.

\bibitem{kian2019a}
{\sc Y.~Kian and L.~Oksanen}, {\em Recovery of time-dependent coefficient on {Riemannian} manifold for hyperbolic equations}, Int. Math. Res. Not., 2019 (2019), pp.~5087--5126.

\bibitem{doi:10.1137/21M1400596}
{\sc V.~P. Krishnan, Rakesh, and S.~Senapati}, {\em Stability for a formally determined inverse problem for a hyperbolic pde with space and time dependent coefficients}, SIAM Journal on Mathematical Analysis, 53 (2021), pp.~6822--6846.

\bibitem{kumar2024stabledeterminationtimedependentmatrix}
{\sc N.~Kumar, T.~Sarkar, and M.~Vashisth}, {\em Stable determination of a time-dependent matrix potential for a wave equation in an infinite waveguide}, arXiv preprint 2409.01913,  (2024).

\bibitem{kurylev2018a}
{\sc Y.~Kurylev, L.~Oksanen, and G.~P. Paternain}, {\em Inverse problems for the connection {Laplacian}}, J. Differ. Geom., 110 (2018), pp.~457--494.

\bibitem{lassas2020}
{\sc M.~Lassas, L.~Oksanen, P.~Stefanov, and G.~Uhlmann}, {\em The light ray transform on {Lorentzian} manifolds}, Commun. Math. Phys., 377 (2020), pp.~1349--1379.

\bibitem{lee2002}
{\sc J.~M. Lee}, {\em Introduction to smooth manifolds}, vol.~218 of Grad. Texts Math., New York, NY: Springer, 2002.

\bibitem{mishra2021determining}
{\sc R.~K. Mishra and M.~Vashisth}, {\em Determining the time-dependent matrix potential in a wave equation from partial boundary data}, Applicable Analysis, 100 (2021), pp.~3492--3508.

\bibitem{oksanen2025interplay}
{\sc L.~Oksanen, G.~P. Paternain, and M.~Sarkkinen}, {\em On the interplay between the light ray and the magnetic {X}-ray transforms}, arXiv preprint 2502.04061,  (2025).

\bibitem{oksanen2024rigiditylorentziancalderonproblem}
{\sc L.~Oksanen, Rakesh, and M.~Salo}, {\em Rigidity in the {L}orentzian {C}alder\'on problem with formally determined data}, arXiv preprint 2409.18604,  (2024).

\bibitem{ONeill}
{\sc B.~O'Neill}, {\em Semi-Riemannian Geometry With Applications to Relativity}, ISSN, Elsevier Science, 1983.

\bibitem{oneill1983}
{\sc B.~O'Neill}, {\em Semi-{Riemannian} geometry. {With} applications to relativity}.
\newblock Pure and {Applied} {Mathematics}, 103. {New} {York}-{London} etc.: {Academic} {Press}. xiii, 468 p. {\$} 45.00 (1983)., 1983.

\bibitem{paternain2023geometric}
{\sc G.~P. Paternain, M.~Salo, and G.~Uhlmann}, {\em Geometric inverse problems}, vol.~204, Cambridge University Press, 2023.

\bibitem{paternain2019}
{\sc G.~P. Paternain, M.~Salo, G.~Uhlmann, and H.~Zhou}, {\em The geodesic {X}-ray transform with matrix weights}, Am. J. Math., 141 (2019), pp.~1707--1750.

\bibitem{ralston1982gaussian}
{\sc J.~Ralston}, {\em Gaussian beams and the propagation of singularities}, Studies in partial differential equations, 23 (1982), p.~C248.

\bibitem{ringstrom2009}
{\sc H.~{Ringstr\"om}}, {\em {The Cauchy problem in general relativity.}}, Z\"urich: European Mathematical Society (EMS), 2009.

\bibitem{stefanov2018}
{\sc P.~Stefanov and Y.~Yang}, {\em The inverse problem for the {Dirichlet}-to-{Neumann} map on {Lorentzian} manifolds}, Anal. PDE, 11 (2018), pp.~1381--1414.

\bibitem{stefanov1989uniqueness}
{\sc P.~D. Stefanov}, {\em Uniqueness of the multi-dimensional inverse scattering problem for time dependent potentials}, Mathematische Zeitschrift, 201 (1989), pp.~541--559.

\bibitem{tataru1995}
{\sc D.~{Tataru}}, {\em {Unique continuation for solutions to PDE's; between H\"ormander's theorem and Holmgren's theorem}}, {Commun. Partial Differ. Equations}, 20 (1995), pp.~855--884.

\bibitem{tetlow2022recovery}
{\sc A.~Tetlow}, {\em Recovery of a time-dependent hermitian connection and potential appearing in the dynamic schro\"odinger equation}, SIAM Journal on Mathematical Analysis, 54 (2022), pp.~1347--1369.

\end{thebibliography}

\end{document}